\documentclass[reqno, final, a4paper]{amsart}
\usepackage{color}
\usepackage[colorlinks=true,allcolors=blue
]{hyperref}
\usepackage{amsmath, amssymb, amsthm}
\usepackage{mathrsfs}
\usepackage{mathtools}
\usepackage[noabbrev,capitalize]{cleveref}
\crefname{equation}{}{}
\usepackage{fullpage}
\usepackage{setspace}
\usepackage{graphics}
\usepackage{pifont}
\usepackage{cancel}
\usepackage{tikz}
\usepackage{bbm}
\usepackage[utf8]{inputenc}
\usepackage[T1]{fontenc}
\usetikzlibrary{arrows.meta}

\usepackage{environ}
\usepackage{paralist}
\usepackage{url}
\usepackage[linesnumbered,ruled,vlined]{algorithm2e}
\usepackage[noend]{algpseudocode}
\usepackage[labelfont=bf]{caption}
\usepackage{framed}
\usepackage[framemethod=tikz]{mdframed}
\usepackage{pgfplots}
\pgfplotsset{compat=1.18}
\usepackage{appendix}
\usepackage{graphicx}
\usepackage[textsize=tiny]{todonotes}
\usepackage{tcolorbox}
\usepackage[usenames,dvipsnames,x11names]{xcolor}
\usepackage[shortlabels]{enumitem}
\crefformat{enumi}{#2#1#3}
\crefrangeformat{enumi}{#3#1#4 to~#5#2#6}
\crefmultiformat{enumi}{#2#1#3}
{ and~#2#1#3}{, #2#1#3}{ and~#2#1#3}
\allowdisplaybreaks

\usepackage{euflag}
\usepackage{cases}
\usepackage{aliascnt}
\usepackage{MnSymbol}

\usepackage[backend = biber, doi = false, url = false, isbn = false, maxbibnames = 222]{biblatex}
\usepackage[mathlines]{lineno}
\usepackage{etoolbox} 

\newcommand*\linenomathpatch[1]{%
	\expandafter\pretocmd\csname #1\endcsname {\linenomath}{}{}%
	\expandafter\pretocmd\csname #1*\endcsname{\linenomath}{}{}%
	\expandafter\apptocmd\csname end#1\endcsname {\endlinenomath}{}{}%
	\expandafter\apptocmd\csname end#1*\endcsname{\endlinenomath}{}{}%
}
\newcommand*\linenomathpatchAMS[1]{%
	\expandafter\pretocmd\csname #1\endcsname {\linenomathAMS}{}{}%
	\expandafter\pretocmd\csname #1*\endcsname{\linenomathAMS}{}{}%
	\expandafter\apptocmd\csname end#1\endcsname {\endlinenomath}{}{}%
	\expandafter\apptocmd\csname end#1*\endcsname{\endlinenomath}{}{}%
}

\expandafter\ifx\linenomath\linenomathWithnumbers
\let\linenomathAMS\linenomathWithnumbers
\patchcmd\linenomathAMS{\advance\postdisplaypenalty\linenopenalty}{}{}{}
\else
\let\linenomathAMS\linenomathNonumbers
\fi

\linenomathpatchAMS{gather}
\linenomathpatchAMS{multline}
\linenomathpatchAMS{align}
\linenomathpatchAMS{alignat}
\linenomathpatchAMS{flalign}
\linenomathpatch{equation}
\linenomathpatch{numcases}
\linenomathpatch{figure}

\renewbibmacro{in:}{}

\newtheorem{thm}{Theorem}[section]

\newaliascnt{cor}{thm}
\newtheorem{cor}[cor]{Corollary}
\aliascntresetthe{cor}

\newaliascnt{lem}{thm}
\newtheorem{lem}[lem]{Lemma}
\aliascntresetthe{lem}

\newaliascnt{fact}{thm}
\newtheorem{fact}[fact]{Fact}
\aliascntresetthe{fact}

\newaliascnt{prop}{thm}
\newtheorem{prop}[prop]{Proposition}
\aliascntresetthe{prop}

\newaliascnt{conj}{thm}
\newtheorem{conj}[conj]{Conjecture}
\aliascntresetthe{conj}

\newaliascnt{proc}{thm}

\aliascntresetthe{proc}

\newtheorem{claim}{Claim}
\newenvironment{proofclaim}[1][Proof of claim]{\begin{proof}[#1]}{\end{proof}}
\theoremstyle{definition}

\newtheoremstyle{break}
{\topsep}{\topsep}%
{}{}%
{\bfseries}{}%
{\newline}{}%
\theoremstyle{break}

\newtheorem{proto-theo}[thm] {Proto-Theorem}   

\theoremstyle{definition}
\newtheorem{rem}[thm]{Remark}

\def\Pr{\mathop{\mathbb{P}}\nolimits}

\let\:=\colon
\let\phi=\varphi

\newcommand{\dist}{\operatorname{d}}

\newcommand{\cC}{\mathcal{C}}

\newcommand{\cG}{\mathcal{G}}

\newcommand{\cP}{\mathcal{P}}
\newcommand{\cQ}{\mathcal{Q}}
\newcommand{\cR}{\mathcal{R}}
\newcommand{\cS}{\mathcal{S}}

\newcommand{\cU}{\mathcal{U}}
\newcommand{\cV}{\mathcal{V}}

\definecolor{ipelightblue}{rgb}{0.678, 0.847, 0.902}

\usepackage{tikz} 
\usepackage{array,color,colortbl}

\renewcommand{\geq}{\geqslant}
\renewcommand{\leq}{\leqslant}

\renewcommand{\emptyset}{\varnothing}
\def\bL{\mathbf{L}}

\providecommand{\col}{\cellcolor{ipelightblue}}

\let\epsilon\varepsilon

\title[Universal probability bounds for partial Latin squares]%
{Universal probability bounds for partial Latin squares}

\author[Allsop]{Jack Allsop}

\address[J. Allsop]{Institut f\"ur Mathematik, Freie Universit\"at Berlin, Germany}
\email{allsop@mi.fu-berlin.de}

\author[Morris]{Patrick Morris}
\address[P. Morris]{Departament de Matem\`atiques, Universitat Polit\`ecnica de Catalunya (UPC),  Barcelona, Spain}
\email{pmorrismaths@gmail.com}

\thanks{J. Allsop was funded by the Deutsche Forschungsgemeinschaft (DFG, German Research 
	Foundation) under Germany's Excellence Strategy – The Berlin Mathematics 
	Research Center MATH+ (EXC-2046/2, project ID: 390685689).}

\thanks{
	P. Morris was supported  by the European Union's Horizon Europe   Marie Sk\l{}odowska-Curie grant RAND-COMB-DESIGN - project number
	101106032 {\euflag} and by a Ram\'on y Cajal fellowship (RYC2024-049272-I).}
    \thanks{
	Part of the research on this project was conducted during the visit of the first  author to UPC Barcelona supported by the Bilateral AEI+DFG Project PCI2024-155080-2: SRC-ExCo- Structure, Randomness and Computational Methods in Extremal
	Combinatorics.}

\date{\today}	

\begin{document}
	\onehalfspace 
	
	\begin{abstract}
		This paper studies the probability of substructures occurring in random Latin squares. Our main result states that if  $\alpha,\beta>0$ are such that $2\alpha+\beta<1$,  then there are positive constants $\delta = \delta(\alpha, \beta)$ and $\Delta = \Delta(\alpha, \beta)$ such that if $P$ is a partial Latin square of order $n$ with $k = k(n)$ non-empty cells occupying at most $\alpha n$ rows and $\beta n$ columns, the probability that a random Latin square of order $n$ contains $P$ lies between $(\delta/n)^k$ and $(\Delta/n)^k$. We apply this result to subsquares in random Latin squares to obtain the first proof of the fact that the expected number of subsquares of order $3$ in a random Latin square of order $n$ is non-vanishing as $n \to \infty$. We are also able to provide the best known asymptotics for the expected number of subsquares of order $a$ in a random Latin square of order $n$ when $2<a=o(n^{1/2})$. Finally, we discuss the implications of our result on other configurations in random Latin squares as well as on completions of partial Latin squares. 
	\end{abstract}
	
	
	\maketitle
	
	\section{Introduction} 
 Many of the techniques developed to handle random combinatorial objects tend to fail for random Latin squares, due to their lack of independence and recursive structure, as well as the difficulty in making small perturbations to a Latin square to create a new one. This makes random Latin squares notoriously difficult objects to study. Despite this, there have been some significant recent advances in the study of random Latin squares \cite{subsqrandomrect,decomp, kwan2020almost,LSparity,KSSS}. However, many fundamental problems regarding random Latin squares are wide open. In this paper, we study the probability of substructures occurring in random Latin squares. Despite being  a natural problem that is fundamental to our understanding of random Latin squares, it has previously only been settled for a small handful of substructures. Here we are able to obtain the right shape of probability for a large class of substructures. 
	
	A \textit{partial Latin square of order $n$} is an $n \times n$ matrix such that each cell is either empty or filled with one of $n$ symbols, with no symbol being repeated in any row or column. A partial Latin square in which every cell is filled is a \textit{Latin square}.   
	Let $P$ be a partial Latin square of order $n$. If $(r, c)$ is a non-empty cell of $P$, then we denote the symbol in row $r$ and column $c$ of $P$ by $P[r, c]$. An \emph{entry} of $P$ is a triple $\left(r, c, P[r, c]\right)$ where $(r, c)$ is a non-empty cell of $P$. We will identify $P$ with the set of its entries. This allows us to use set notation such as $(1, 2, 3) \in P$, meaning that $P[1, 2] = 3$. We thus use $|P|$ to denote the number of non-empty cells of $P$ and we say that a Latin square $L$ of order $n$ \emph{contains} $P$, denoted $L \supseteq P$, if it contains each entry of $P$.
	
	For a partial Latin square $P$ of order $n$, we let $\cR_P$ be the set of rows featuring non-empty cells of $P$, $\cC_P$ the set of columns containing non-empty cells of $P$, and $\cS_P$ the set of symbols appearing in $P$.  We can now state our main result.
	
	\begin{thm} \label{t:main}
		Let $\bL$ be a random Latin square of order $n$ and let $P$ be a partial Latin square of order $n$. Let $\alpha, \beta > 0$ be such that $|\cR_P| \leq \alpha n$ and $|\cC_P| \leq \beta n$ and suppose that $2\alpha+\beta<1$. Then
		\[ 
		\left(\frac{\delta}{n}\right)^{|P|}\leq \Pr(\mathbf{L}\supseteq P)\leq \left(\frac{\Delta}{n}\right)^{|P|}
		\]
		for positive constants $\delta=\delta(\alpha,\beta)$ and $\Delta=\Delta(\alpha,\beta)$ that satisfy $\delta\rightarrow 1/23$ and $\Delta\rightarrow 23$ as $\alpha, \beta \rightarrow 0$.  
	\end{thm}

\begin{rem} 
   It is possible to permute the roles of rows, columns and symbols in  Latin squares. For instance, for each Latin square $L$ there is a Latin square $\tilde{L}:=\{(s,c,r):(r,c,s)\in L\}$ obtained by switching the roles of the rows and the symbols. Using such permutations, one can adjust our proof to hold under the  more general condition that there exist two distinct sets $\cU,\cV\in \{\cR_P,\cC_P,\cS_P\}$ such that $|\cU|\leq \alpha n$ and $|\cV|\leq \beta n$. 
\end{rem}

We will return to discuss the tightness of the constants $\delta,\Delta>0$ and the condition $2\alpha+\beta<1$ in Section \ref{ss:tightness}. First though, we present some implications of Theorem \ref{t:main}.

\subsection{Latin subsquares}
By far the most studied partial Latin squares in random Latin squares have been Latin subsquares. A \emph{subsquare} of a Latin square $L$ is a partial Latin square $P\subseteq L$  whose non-empty cells form a Latin square. A \emph{proper subsquare} of $L$ is a subsquare whose order lies strictly between $1$ and $n$. The largest possible order of a proper subsquare of $L$ is $n/2$. An \emph{intercalate} is a subsquare of order $2$. 
A major goal in this area, initiated by McKay and Wanless \cite{manysubsq} in 1999,  has been to determine (asymptotically) the expected number of subsquares of a given order $a$ in a random Latin square of order $n$. We denote this quantity by $\mathbb{E}_a(n)$. They made the following conjecture. 

\begin{conj}[McKay and Wanless \cite{manysubsq}] \label{conj:MW}
    For $2\leq a\leq n/2$, one has 
    \begin{subnumcases}
        {\mathbb{E}_a(n)=}
       \frac{n^2}{4}(1+o(1)) & if  $a=2$; \label{MW2}\\ 
        \frac{1}{18}(1+o(1)) & if  $a=3$; \label{MW3} \\
        o(1) & if $4\leq a\leq n/2$. \label{MW4+}
        \end{subnumcases}  
\end{conj}

 They also provided an initial estimate of $\mathbb{E}_2(n)\geq n^{3/2-o(1)}$ \cite{manysubsq}. Cavenagh, Greenhill and Wanless \cite{cycstrucrandom} then provided an upper bound of $\mathbb{E}_2(n)\leq \frac{9}{2}n^{5/2}$ in 2008. In 2018, Kwan and Sudakov \cite{KS} showed that $\frac{n^2}{4}(1-o(1))\leq \mathbb{E}_2(n)\leq \frac{n^2}{2}(1+o(1))$ before part \eqref{MW2} of Conjecture \ref{conj:MW} was completely resolved by Kwan, Sah and Sawhney \cite{KSS} in 2022.   The same authors with Simkin \cite{KSSS} then provided tight tail bounds for the count of intercalates in random Latin squares, thus providing a complete picture for this case. Until recently, there was little progress on part \cref{MW4+} of Conjecture \ref{conj:MW} with the exception of a result of McKay and Wanless \cite{manysubsq} themselves who proved the statement when $a=n/2$. In a recent flurry of results,  Divoux, Kelly, Kennedy and Sidhu \cite{subsqrandom} and independently Gill, Mammoliti and Wanless \cite{canonicallabel} extended the range to $a\geq  n^{1/2}(\log n)^{1/2+o(1)}$  (the former result being slightly stronger by requiring some large constant in place of the  $(\log n)^{o(1)}$ factor). Finally Allsop and Wanless \cite{subsqrandomrect} settled part \eqref{MW4+} of Conjecture \ref{conj:MW}  by showing that $\mathbb{E}_a(n)=O(n^{-2})$ for $4\leq a\leq n/2$. Using a union bound and Markov's inequality, this in fact shows that asymptotically almost surely (a.a.s.\ for short), that is, with probability tending to 1 as $n$ tends to infinity, one has that a random Latin square $\bL$ does not contain any $a \times a$ subsquare with $a\geq 4$. 

 As predicted by McKay and Wanless \cite{manysubsq}, part \eqref{MW3} of Conjecture \ref{conj:MW} has been the most challenging part  and remains unresolved to date. Allsop and Wanless \cite{subsqrandomrect} were the first to give a constant upper bound on $\mathbb{E}_3(n)$, showing that $\mathbb{E}_3(n)\leq \frac{2}{3}+o(1)$. As an application of Theorem \ref{t:main}, we complement this by proving that $\mathbb{E}_3(n) = \Omega(1)$. This  marks significant progress towards a resolution of Conjecture~\ref{conj:MW} \eqref{MW3} and such a lower bound was highlighted as an open problem in the thesis of the first author~\cite{thesis}.

	\begin{thm}\label{t:E3n}
		\[
		\mathbb{E}_3(n) \geq \frac1{219006}-o(1).
		\]
	\end{thm}

	 We are also able to provide estimates on the asymptotics for $\mathbb{E}_a(n)$ for larger $a$ which significantly improve the asymptotics of Allsop and Wanless \cite{subsqrandomrect} in the range $4\leq a\leq n/2$.
	
	\begin{thm}\label{t:Emn_asymptotics}
		Suppose that $2\leq a \leq n/2$. Then
		\[
		\left(\frac{1-o(1)}{23e^2}\right)^{a^2}\left(\frac an\right)^{a^2-3a} \leq \mathbb{E}_a(n) \leq \left(\frac{23+o(1)}{e^{1/2}}\right)^{a^2}\left(\frac an\right)^{a^2-3a}.
		\]
	\end{thm}
	
	We remark that we use Theorem \ref{t:main} only for the range where $a=o(n)$; for larger $a$ we can achieve Theorem \ref{t:Emn_asymptotics} by using known estimates on permanents\footnote{A standard technique in this area is to use estimates on permanents of certain matrices to obtain information about Latin squares. See, for example,~\cite[Section $17$]{num_squares}.} and the number of Latin rectangles of certain sizes, as used for example in \cite{canonicallabel,manysubsq}.
    In fact,  when $a\geq \omega(n^{1/2}\log n)$ by using either these permanent estimates  or a result of Divoux, Kelly, Kennedy and Sidhu \cite{subsqrandom}, one can get matching constants in the upper and lower bounds of \cref{t:Emn_asymptotics} which depend on $a$ (see Proposition \ref{prop:rectangles}). 
   
    \subsection{Latin subrectangles} 
Our results generalise to the setting of Latin subrectangles in random Latin squares which to our knowledge has not been previously studied (although we believe it to be very natural). 
Throughout this subsection, let $a = a(n)$ and $b = b(n)$ be positive integer functions of $n$ such that $a \leq b \leq n$. An $a \times b$ \emph{partial Latin rectangle} is an $a \times b$ array where each cell is either empty or filled with one of $b$ symbols such that no symbol appears multiple times in any row or column. A \emph{Latin rectangle} is a partial Latin rectangle with no empty cells. Given a Latin square $L$ of order $n$, a \textit{subrectangle} of $L$ is a submatrix of $L$ that is itself a Latin rectangle. Let $\mathbb{E}_{a, b}(n)$ denote the expected number of $a \times b$ subrectangles in a random Latin square of order $n$.  We can use \cref{t:main} to give bounds on $\mathbb{E}_{a, b}(n)$ as follows. 

\begin{thm}\label{t:Eabn}
    For $\varepsilon> 0$   and $a+b \leq (1-\varepsilon) n$ there exist  $\delta'=\delta'(\varepsilon),\Delta'=\Delta'(\varepsilon)>0$   such that 
		\[
		\delta'^{ab} \left(\frac bn\right)^{ab-2b-a}\left(\frac ba\right)^a\leq \mathbb{E}_{a, b}(n) \leq \Delta'^{ab} \left(\frac bn\right)^{ab-2b-a}\left(\frac ba\right)^a.
		\]
\end{thm}

In fact, we can get matching constants $\delta'=\Delta'$ in the case where $a = \Omega(n)$ in \cref{t:Eabn}.
Using the upper bound of \cref{t:Eabn}, Markov's inequality and a union bound, we obtain the following strengthening  of the result of Allsop and Wanless. 

\begin{cor} \label{cor:no_rectangles} A.a.s.\ one has that a random Latin square $\bL$ does not contain any $a \times b$ subrectangles with $3\leq a\leq b$ and $(a,b)\neq (3,3)$. 
    \end{cor}

\subsection{Isotopic copies of partial Latin squares} \label{ss:patterns}
Moving away from subsquares and subrectangles, we can look more generally for \textit{copies} of any partial Latin square in random Latin squares. Two partial Latin squares $P$ and $P'$ of order $n$ are said to be \textit{isotopic} if there exist permutations of the rows, columns and symbols which result in $P$ being mapped to $P'$ (here we permute the whole row set $[n]$, column set $[n]$ and symbol set $[n]$, not just those with non-empty entries of $P$). We remark that not all subsquares (or subrectangles) of a given size are isotopic but the results of the previous subsections rather study the union of all the isotopism classes of Latin subsquares (or subrectangles) of a given size.  In general, isotopisms preserve many properties of partial Latin squares (for example completability, cf.\ Section \ref{ss:tightness}) and so it is interesting to determine whether there exists an isotopic copy of a given partial Latin square in a random Latin square. This could be used, for example, to determine the existence of certain \textit{Latin trades} in random Latin squares, which are collections of entries that can be removed and replaced by another set of entries to obtain a new Latin square. See~\cite{trade_survey} for a survey on Latin trades.
For convenience, in this subsection, we restrict to looking at constant-sized partial Latin squares, although we expect the same behaviour to hold for partial Latin squares whose number of entries is (mildly) growing with $n$. 

Inspired by the foundational work uncovering thresholds for small subgraphs in random (hyper-)graphs (see for example \cite[Chapter 3]{JLRbook}), we make the following definitions. For a non-empty partial Latin square $P$, 
we define the \textit{density}  to be 
\[m(P):=\min\left\{\frac{|\cR_{P'}|+|\cC_{P'}|+|\cS_{P'}|-|P'|}{|P'|}:\emptyset \neq P'\subseteq P\right\}. \]

Note also  that $-1<m(P)\leq 2$ for all partial Latin squares $P$ as $0<|\cR_{P'}|,|\cC_{P'}|,|\cS_{P'}|\leq |P'|$ for any non-empty partial Latin square $P'$. By considering transversals\footnote{A \textit{transversal} is a partial Latin square with no repeated row, column or symbol amongst its non-empty cells.}, one can get partial Latin squares  achieving the upper bound whilst partial Latin squares $P$ whose non-empty cells form $a\times a$ Latin squares have $m(P)=3/a-1$, showing that $m(P)$ can get arbitrarily close to $-1$. We make the following conjecture which provides a dichotomy for copies of partial Latin squares appearing in random Latin squares, depending on their density. 

\begin{conj} \label{conj:copies}
    For $P$ a partial Latin square and $\bL$ a random Latin square of order $n$, one has that as $n\rightarrow \infty$,
    \begin{subnumcases}
        {\Pr(\bL \textnormal{ contains an isotopic copy of   }P)\rightarrow }
      0 & if  $m(P)<0$; \label{Cneg}\\ 
        1 & if $m(P)>0$. \label{Cpos}
        \end{subnumcases}  
\end{conj}
Moreover when $m(P)=0$, we expect the number of copies of $P$ in $\bL$ to have a Poisson distribution. This is conjectured to be the case, for example, with $3 \times 3$ subsquares~\cite{KSS}. Note that the definition of $m(P)$ takes the minimum over all non-empty $P'\subseteq P$ here which is important. There are partial Latin squares $P$ such that $|\cR_{P}|+|\cC_{P}|+|\cS_{P}|-|P|>0$ and yet $m(P)<0$, such as a subsquare of order 4 with 3 additional entries each with unique row, column and symbol.

 Using Theorem \ref{t:main}, we can easily derive part \eqref{Cneg} of Conjecture \ref{conj:copies} by the first moment method. Indeed, if $m(P)<0$, then there is some $P'\subseteq P$ such that $t:=|\cR_{P'}|+|\cC_{P'}|+|\cS_{P'}|-|P'|<0$. Now the number of isotopic copies of $P'$ is at most $n^{|\cR_{P'}|+|\cC_{P'}|+|\cS_{P'}|}$ by considering the choices of rows, columns and symbols to map to. For each such isotopic copy, the probability that it lies in $\bL$ is at most, say, $(24/n)^{|P'|}$ when $n$ is sufficiently large, by Theorem \ref{t:main}. Hence, by a union bound, the probability that $\bL$ contains an isotopic copy of $P'$ is at most $24^{|P'|}n^{t}$ which tends to $0$ as $n$ tends to infinity due to the fact that $t<0$. Hence a.a.s.\ $\bL$ contains no copy of $P'$ and hence certainly no copy of $P$.  

 In the other direction, we give the following positive result. 
 \begin{prop} \label{prop:partial}
     If $P$ is a partial Latin square with $m(P)>1/2$, then a.a.s.\ a random Latin square $\bL$ of order $n$ will contain an isotopic copy of $P$. 
 \end{prop}

We remark that Proposition \ref{prop:partial} could also be proven using powerful general purpose tools \cite{kwan2020almost,KSS,LSparity} that are able to transfer events that hold with very high probability in a random triangle removal process to a random Latin square.  In fact, such results can prove the asymptotic almost sure existence of copies of many more partial Latin squares $P$ than those covered by Proposition \ref{prop:partial} (although not all $P$ as in \eqref{Cpos}). Nonetheless, we include the short proof of Proposition \ref{prop:partial} as a simple example of how Theorem \ref{t:main} can be applied in conjunction with the second moment method. 

\subsection{Tightness and completions} \label{ss:tightness}
It is natural to ask what the best possible values of $\delta>0$ and $\Delta>0$ in Theorem \ref{t:main} could be. For `small' Latin squares, it is likely that both constants can be $(1+o(1))$, that is, that $((1-o(1))/n)^{|P|}\leq \Pr(\mathbf{L}\supseteq P)\leq ((1+o(1))/n)^{|P|}$. This would have several implications. In particular, for constant sized $P$,  it would settle part \eqref{MW3} of Conjecture \ref{conj:MW} and Conjecture \ref{conj:copies} (as well as being able to show a Poisson distribution for the number of copies of $P$ with $m(P)=0$). Divoux, Kelly, Kennedy and Sidhu \cite{subsqrandom} in fact conjectured such probability bounds for a very general notion of `small' partial Latin squares, namely those that are locally bounded. We will return to discuss this in more detail in the concluding remarks (Section \ref{s:conc}). 

Nonetheless, one of the strengths of Theorem \ref{t:main} is that it also holds for partial Latin squares that are not small. For such Latin squares, it is no longer the case that 1 is the right constant. Indeed, as we will show in Proposition \ref{prop:rectangles}, it follows from known estimates on permanents  that there is some $c=c(\alpha)\neq 1$ such that for all $0<\alpha <1/2$, the probability of containing a Latin subsquare of order $\alpha n$ lies between $((c(\alpha)-o(1))/n)^{\alpha^2n^2}$ and $((c(\alpha)+o(1))/n)^{\alpha^2n^2}$. On the other hand, in \cref{s:conc} we give an example of a partial Latin square $P'$ with $\alpha n$ rows and $\alpha n$ columns  with $\Pr(\bL \supseteq P')=((c'(\alpha)+o(1))/n)^{\alpha^2 n^2}$ for $c'(\alpha)\neq c(\alpha)$. Therefore one cannot hope for a statement like Theorem \ref{t:main} with a universal constant  $\delta=\Delta$ depending only on $|\cR_P|/n$ and $|\cC_P|/n$.     The  correct constant for a given partial Latin square $P$ will depend on the structure of $P$ and is somewhat mysterious, as we will discuss further in the concluding remarks.

In addition to assessing the quantitative tightness of the constants $\delta,\Delta>0$ in Theorem \ref{t:main}, it is also interesting to reflect on the qualitative nature of the bounds. That is, for which partial Latin squares $P$ can we expect probability bounds of the shape given in Theorem \ref{t:main}?  For this, it is helpful to introduce the notion of \textit{completions}. A \emph{completion} of a partial Latin square is a Latin square that contains it. If a partial Latin square has a completion, then it is \emph{completable}, otherwise it is \emph{non-completable}. Completions of partial Latin squares has been a well-studied topic  with particular interest in whether certain partial Latin squares are completable or not. See~\cite{completion_survey} for a survey on this topic. For example,  a classical result of Ryser \cite{ryser1951combinatorial} states that if $a+b\leq n$, then any $a\times b$ subrectangle is completable to a Latin square of order $n$, whilst for all $a, b$ with $a+b>n$, there are examples of non-completable $a\times b$ subrectangles (which shows in particular the asymptotic tightness of the condition on $a+b$ in Theorem \ref{t:Eabn}).

Note that $\mathbb{P}(\mathbf{L} \supseteq P) = 0$ for a random Latin square $\mathbf{L}$ precisely when $P$ is non-completable.  In contrast, using that there are $((1\pm o(1))n/e^2)^{n^{2}}$ Latin squares of order $n$ \cite{num_squares}, the lower bound of Theorem \ref{t:main} gives  $(cn)^{n^2-|P|}$ completions for some $c=c(P)>0$. This is a stark difference and shows the need for extra conditions on $P$ in order to derive the lower bounds in Theorem \ref{t:main}. For us, this is the condition $2\alpha+\beta<1$ which is asymptotically tight in the following sense: For all $\epsilon > 0$ and sufficiently large $n$, there exists a non-completable partial Latin square $P$ of order $n$ with $|\mathcal{R}_P| \leq \alpha n$ and $|\mathcal{C}_P| \leq \beta n$ where $2\alpha+\beta = 1+\epsilon$. To see this, set $\alpha = 3\epsilon/5$, $\beta = 1-\epsilon/5$, and $\gamma = 1-2\epsilon/5$ so that $2\alpha+\beta = 1+\epsilon$. Let $S_1 = \{1, 2, \ldots, \lfloor\gamma n\rfloor\}$ and let $S_2 = \{\lfloor\gamma n\rfloor+1, \lfloor\gamma n\rfloor+2, \ldots, n\}$. Let $P$ be the partial Latin square of order $n$ defined by: 
    \begin{itemize}
        \item $P[1, i] = i$ for all $i \in S_1$,
        \item $P[i, n] = i$ for all $i \in S_2$, and
        \item all other cells of $P$ are empty.
    \end{itemize}
    Then $|\mathcal{R}_P| = 1+n-\lfloor \gamma n\rfloor \leq 2+2\epsilon n/5 \leq \alpha n$ and $|\mathcal{C}_P| = 1+\lfloor \gamma n\rfloor \leq 1+(1-2\epsilon/5)n \leq \beta n$.
    We claim that $P$ is non-completable. Suppose, for a contradiction, that $L$ is a Latin square of order $n$ that contains $P$ and let $s = L[1, n]$. Observe that $s \notin S_1$, since each symbol in $S_1$ already occurs in row $1$, and also observe that $s \notin S_2$, since each symbol in $S_2$ already occurs in column $n$. But $s$ must be an element of $[n]=S_1 \cup S_2$, which is a contradiction.
	This shows  a striking dichotomy either side of this condition, with the possibility of no completions when $2\alpha+\beta>1$, whilst any partial Latin square has many completions when $2\alpha+\beta<1$.
    It would be interesting to determine further conditions on partial Latin squares under which the lower bound of Theorem \ref{t:main} holds. As a first step, we  suspect that our  condition  can be sharpened to $\alpha+\beta < 1$, where the dichotomy would become even more apparent and there is a richer family of examples demonstrating it. Finally we mention that an upper bound of the shape of Theorem \ref{t:main} is believed to hold for all partial Latin squares without any restrictions, as we discuss in the concluding remarks (Section \ref{s:conc}).

	\subsection{Notation}\label{s:notat}	
	
	Throughout this paper, $n$ will be a positive integer and unless stated otherwise, all asymptotics will be as $n \to \infty$. Unless otherwise stated, all probability distributions will be discrete and uniform with $\Pr(\cdot)$ denoting probability. Random objects will be printed in bold. For $\{a, b\} \subseteq \mathbb{Z}$, we will use $[a, b]$ to denote the integer closed interval $\{z \in \mathbb{Z} : a \leq z \leq b\}$. We use $[b]$ as shorthand for $[1, b]$. Unless specified otherwise, the row index set, column index set, and symbol set of a partial Latin square of order $n$ will be $[n]$.
		
	\subsection{Organisation} \label{s:org}

    The rest of this paper is organised as follows. Our proof of \cref{t:main} uses the switching method. For our proof, we will require several different switching procedures, which we define in \cref{s:switch}. 
    It turns out that in order to effectively use our switching procedures to prove \cref{t:main}, we need to be able to bound the probability that a random Latin square $\bL$ contains  an intercalate, when we condition on the event that $\bL$ contains a partial Latin square satisfying the hypotheses of \cref{t:main}. Such a bound is proved in \cref{s:interc}. We are then able to prove \cref{t:main} in \cref{s:main}. In \cref{s:applications}, we derive the various consequences of \cref{t:main} proving Theorems \ref{t:E3n}, \ref{t:Emn_asymptotics} and \ref{t:Eabn} as well as Corollary \ref{cor:no_rectangles} and Proposition \ref{prop:partial}.  Finally, we discuss the drawbacks of \cref{t:main} and provide some possible directions for future research in \cref{s:conc}.

\subsubsection*{Acknowledgements} We would like to thank Guillem Perarnau and Matthew Kwan 
for interesting conversations around the topic of this paper.

	\section{Switching procedures}\label{s:switch}
	
	The switching method underlies many of the results regarding substructures in random Latin squares, including those in~\cite{subsqrandomrect, cycstrucrandom, subsqrandom, canonicallabel, GM90, KSS, KSSS, KS, manysubsq}. 
    For more details  on the  background of this method, see~\cite{genswitch}. 
	
	\subsection{Cycle switching} \label{s:cycle switching}
	
	Our first switching procedure is known as cycle switching. Let $L$ be a Latin square of order $n$. Let $a$ and $b$ be positive integers satisfying $a \leq b \leq n$. An $a \times b$ \emph{subrectangle} of $L$ is an $a \times b$ submatrix that is itself a Latin rectangle. An $a \times b$ subrectangle is called \emph{proper} if $1 < a \leq b < n$.
	A \emph{row cycle} of length $\ell$ in $L$ is a $2 \times \ell$ subrectangle that does not contain any proper subrectangle. An \emph{intercalate} is a row cycle of length $2$. Equivalently, an intercalate is a submatrix of the form
    \[
    \begin{pmatrix}
        s&s'\\
        s'&s
    \end{pmatrix}
    \]
    in a Latin square.
    
    Let $\{r, r', c\} \subseteq [n]$ with $r \neq r'$. There is a unique row cycle of $L$ that hits rows $r$ and $r'$ and column $c$, which we denote by $\rho_L(r, r', c)$. Let $\cC$ be the set of columns hit by $\rho_L(r, r', c)$. We can define a new Latin square $L'$ by
	\[
	L'[i, j] = \begin{cases}
		L[r', j] & \text{if } i = r \text{ and } j \in \cC,\\
		L[r, j] & \text{if } i = r' \text{ and } j \in \cC,\\
		L[i, j] & \text{otherwise}.
	\end{cases}
	\]
	We say that $L'$ has been obtained from $L$ by switching on $\rho_L(r, r', c)$. See Figure~\ref{f:row_cyc_switch} for an example of row cycle switching.
	
	\begin{figure}[ht]
		\[
		\begin{pmatrix}
			1&2&3&4&5&6&7\\
			\col2&\col7&1&\col5&3&4&\col6\\
			3&1&2&7&6&5&4\\
			4&3&6&1&2&7&5\\
			5&6&7&2&4&1&3\\
			\col7&\col5&4&\col6&1&3&\col2\\
			6&4&5&3&7&2&1\\
		\end{pmatrix}
		\quad\longrightarrow\quad
		\begin{pmatrix}
			1&2&3&4&5&6&7\\
			\col7&\col5&1&\col6&3&4&\col2\\
			3&1&2&7&6&5&4\\
			4&3&6&1&2&7&5\\
			5&6&7&2&4&1&3\\
			\col2&\col7&4&\col5&1&3&\col6\\
			6&4&5&3&7&2&1\\
		\end{pmatrix}
		\]
		\caption{\label{f:row_cyc_switch}The highlighted entries in the Latin square $L$ on the left form the row cycle $\rho = \rho_L(2, 6, 1)$, which has length $4$. The Latin square on the right is obtained from $L$ by switching on $\rho$.}
	\end{figure}

	By definition, one  immediately obtains that row cycle switches are \textit{reversible}.
	
	\begin{fact}[Reversibility of row cycle switches] \label{f:row_cyc_reverse}
		Let $L$ be a Latin square of order $n$ and let $\{r, r', c\} \subseteq [n]$ with $r \neq r'$. 
		If $L$ is obtained from some Latin square $L'$ by switching on $\rho_{L'}(r, r', c)$, then $L'$ is uniquely determined: it is the Latin square obtained from $L$ by switching on $\rho_L(r, r', c)$.
	\end{fact}

	We will also need the analogous notion of column cycle switching. Let $L^*$ denote the transpose of $L$. A \emph{column cycle} of length $\ell$ in $L$ is the image of a row cycle in $L^*$ under transposition. Equivalently, a column cycle in $L$ is an $\ell \times 2$ submatrix that contains exactly $\ell$ symbols and that does not contain, for any $\ell' < \ell$, a $2 \times \ell'$ submatrix with exactly $\ell'$ symbols. For any $\{c, c', r\} \subseteq [n]$ with $c \neq c'$, there is a unique column cycle that hits columns $c$ and $c'$ and row $r$, which we denote by $\gamma_L(c, c', r)$. We can switch on column cycles in an analogous way to how we switched on row cycles. Again, we immediately get that this procedure is reversible.
	
	\begin{fact}[Reversibility of column cycle switches] \label{f:col_cyc_reverse}
		Let $L$ be a Latin square of order $n$ and let $\{c, c', r\} \subseteq [n]$ with $c \neq c'$. 
		If $L$ is obtained from some Latin square $L'$ by switching on $\gamma_{L'}(c, c', r)$, then $L'$ is uniquely determined: it is the Latin square obtained from $L$ by switching on $\gamma_L(c, c', r)$.
	\end{fact}
	
	There is a third type of cycle switching which is \textit{symbol cycle switching}. Define $\tilde L$ to be the image of $L$ under the map defined by swapping the roles of symbols and rows in $L$. That is, $\tilde L:=\{(L[r,c],c,r):(r,c,L[r,c])\in L\}$. A symbol cycle of length $\ell$ in $L$ is the image of a row cycle in $\tilde L$ under this map.    
    We will not work with symbol cycle switches directly and so we will not elaborate on their properties, all of which can also be derived from considering the analogous results for row cycle switches in $\tilde{L}$.  
	We collectively refer to row cycle switching, column cycle switching  and symbol cycle switching as \textit{cycle switching}. Despite its simplicity, cycle switching has been used to prove many important results regarding substructures in random Latin rectangles. For instance, cycle switching was the main tool used to prove the results in~\cite{subsqrandomrect}. See~\cite{cycswitch} for a study on cycle switching in Latin squares.

	Let $\{r, r'\} \subseteq [n]$ with $r \neq r'$. The permutation mapping row $r$ of $L$ to row $r'$, denoted by $\pi^\rho_{r, r'}$, is defined by $\pi^\rho_{r, r'}(L[r, c]) = L[r', c]$ for all $c \in [n]$. Such a permutation is called a \emph{row permutation} of $L$. When we write a row permutation in disjoint cycle notation, each cycle corresponds to a row cycle of $L$ involving rows $r$ and $r'$. Similarly, if  $\{c, c'\} \subseteq [n]$ with $c \neq c'$, then the \emph{column permutation} mapping column $c$ to $c'$, denoted by $\pi^\gamma_{c, c'}$, is defined by $\pi^\gamma_{c, c'}(L[r, c]) = L[r, c']$ for all $r \in [n]$. A cycle in $\pi^\gamma_{c, c'}$ corresponds to a column cycle of $L$ hitting columns $c$ and $c'$.

	The following lemma, proven by Cavenagh, Greenhill and Wanless \cite{cycstrucrandom},  shows how performing a column cycle switch affects the length of a row cycle  in the Latin square.
	
	\begin{lem}[The effect of a column cycle switch on a row cycle]\label{l:col_cyc_switch}
		Let $L$ be a Latin square of order $n$. Let $\{r, r', c\} \subseteq [n]$ with $r \neq r'$. Write the cycle of $\pi^\rho_{r, r'}$ containing symbol $L[r, c]$ as $(z_1 = L[r, c], z_2, \ldots, z_m)$ and for $i \in [m]$, let $c_i$ be the column such that $L[r, c_i] = z_i$. Now let $c'\neq c''\in [n]$ such that  $\gamma_L(c', c'', r)$ does not hit row $r'$. Let $L'$ be the Latin square obtained from $L$ by switching on $\gamma_L(c', c'', r)$. Then 
		\begin{enumerate}
			\item \label{i:row both intersect} If $c'=c_i$ and $c''=c_j$ for some $\{i, j\} \subseteq [m]$ with $1<i < j$, then the set of columns hit by $\rho_{L'}(r, r', c)$ is $\{c_\ell : 1 \leq \ell \leq i-1 \text{ or } j \leq \ell \leq m\}$;
			\item \label{i:row one intersect} If $c'=c_i$ for some $i\in [m]$ and $c''\notin\{c_1,\ldots,c_m\}$, then the set of columns hit by $\rho_{L'}(r, r', c)$ is the union of the set of columns hit by $\rho_{L}(r, r', c)$ and the set of columns hit by $\rho_L(r,r',c'')$; 
			\item \label{i:row no intersect}  If $c',c''\notin \{c_1,\ldots,c_m\}$ then $\rho_{L'}(r,r',c)=\rho_L(r,r',c)$ is unchanged. 
		\end{enumerate}
	\end{lem}

	\begin{figure}[h]
		\centering
		\includegraphics[width=0.84\linewidth]{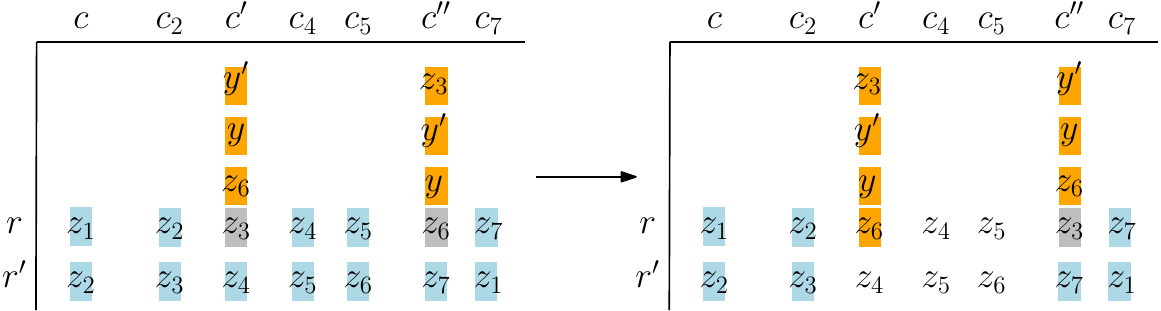}
		\caption{An example of a row cycle $\rho_L(r,r',c)$ (in blue/grey) before and after a switch along column cycle $\gamma_L(c',c'',r)$  (in orange/grey) as in Lemma \ref{l:col_cyc_switch} \eqref{i:row both intersect} with $m=7$,  $c'=c_3$ and $c''=c_6$.   }
		\label{fig:row cyc after col switch}
	\end{figure}
	
	An example of the first situation \eqref{i:row both intersect} is shown in \cref{fig:row cyc after col switch}. We remark that Lemma \ref{l:col_cyc_switch} deals entirely with the case that $\gamma_L(c',c'',r)$ does not intersect row $r'$. In the case that $\gamma_L(c',c'',r)$ does hit row $r'$, performing a switch on $\gamma_L(c',c'',r)$ will not affect the length of $\rho_L(r,r',c)$ or the row permutation $\pi^\rho_{r, r'}$, although $\rho_{L'}(r,r',c)$ will contain different entries to $\rho_L(r,r',c)$ if $\{c',c''\}\cap \{c_1,\ldots,c_m\}\neq \emptyset$.  
	
	Analogously, switching along a row cycle can affect column cycles in a Latin  square. We collect the consequences in the following lemma.

	\begin{lem}[The effect of a row cycle switch on a column cycle]\label{l:row_cyc_switch}
		Let $L$ be a Latin square of order $n$. Let $\{c, c', r\} \subseteq [n]$ with $c \neq c'$. Write the cycle of $\pi^\gamma_{c, c'}$ containing symbol $L[r, c]$ as $(y_1 = L[r, c], y_2, \ldots, y_m)$ and for $i \in [m]$, let $r_i$ be the row such that $L[r_i, c] = y_i$. Now let $r'\neq r''\in [n]$ such that  $\rho_L(r', r'', c)$ does not hit column $c'$. Let $L'$ be the Latin square obtained from $L$ by switching on $\rho_L(r', r'', c)$. Then 
		\begin{enumerate}
			\item \label{i:col both intersect} If $r'=r_i$ and $r''=r_j$ for some $\{i, j\} \subseteq [m]$ with $1<i < j$, then the set of rows hit by $\gamma_{L'}(c, c', r)$ is $\{r_\ell : 1 \leq \ell \leq i-1 \text{ or } j \leq \ell \leq m\}$; 
			\item \label{i:col one intersect} If $r'=r_i$ for some $i\in [m]$ and $r''\notin\{r_1,\ldots,r_m\}$, then the set of rows hit by $\gamma_{L'}(c, c', r)$ is the union of the set of rows hit by $\gamma_{L}(c, c', r)$ and the set of rows hit by $\gamma_L(c,c',r'')$;
			\item \label{i:col no intersect}  If $r',r''\notin \{r_1,\ldots,r_m\}$ then $\gamma_{L'}(c,c',r)=\gamma_L(c,c',r)$ is unchanged. 
		\end{enumerate}
	\end{lem}

	\begin{figure}[h]
		\centering
		\includegraphics[width=0.64\linewidth]{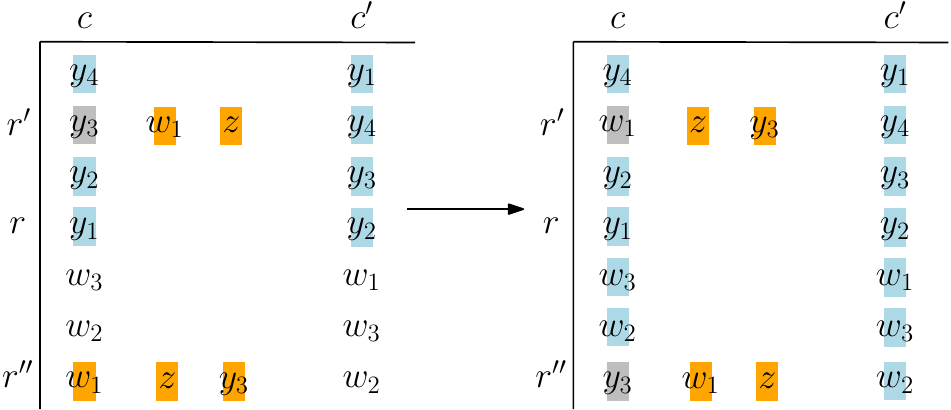}
		\caption{An example of a column cycle $\gamma_L(c,c',r)$ (in blue/grey) before and after a  switch along row cycle $\rho_L(r',r'',c)$  (in orange/grey) as in Lemma \ref{l:row_cyc_switch} \eqref{i:col one intersect} with $m=4$ and $r'=r_3$.   }
		\label{fig:col cyc after row switch}
	\end{figure}

	\subsection{Partial cycles} \label{s:partial}
	Let $\{r, r'\} \subseteq [n]$ with $r \neq r'$ and let $\pi^\rho_{r, r'}$ be the row permutation  of $L$ defined by these rows. Furthermore, let $c$ be a column of $L$ and let $c' \neq c$ be a column that is hit by the row cycle $\rho_L(r, r', c)$. Let $s = L[r, c]$ and let $s' = L[r, c']$. Write the cycle of $\pi^\rho_{r, r'}$ containing $s$ as $(z_1 = s, z_2, \ldots, z_m = s', z_{m+1}, \ldots, z_\ell)$. For $i \in [m]$, let $c_i$ be the column such that $L[r, c_i] = z_i$. Note that $c = c_1$ and $c' = c_m$. Define $\rho_L(r, r', c\rightarrow c')$ to be the set of entries 
	\[
	\{(r, c_i, z_i), (r', c_i, z_{i+1}) : i \in [m-1]\}.
	\]
	We say that $\rho_L(r, r', c\rightarrow c')$ is the \emph{partial row cycle} from $c$ to $c'$ along $\rho_L(r, r', c)$. We define the \emph{distance} from $c$ to $c'$ along $\rho_L(r, r', c)$, denoted by $\dist_L^\rho(r, r', c\rightarrow c')$, to be the number of columns hit by $\rho_L(r, r', c\rightarrow c')$ (which is $m-1$ in the notation used above). We extend this distance function to all $(r, r', c, c') \in [n]^4$ with $r \neq r'$ and $c \neq c'$ by setting $\dist_L^\rho(r, r', c\rightarrow c') = \infty$ if $\rho_L(r, r', c)$ does not hit column $c'$.
	
	\begin{figure}[h]
		\centering
		\includegraphics[width=0.54\linewidth]{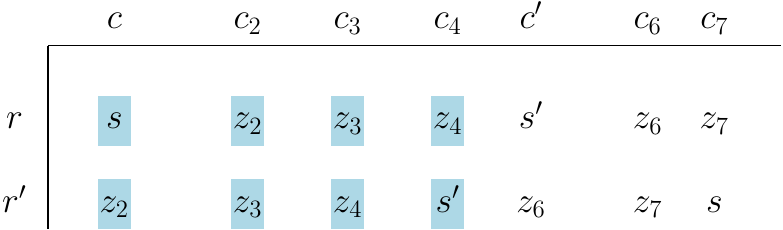}
		\caption{A partial row cycle $\rho_L(r,r',c\rightarrow c')$ with $\dist_L^\rho(r,r',c\rightarrow c')=4$.   }
		\label{fig:partial row}
	\end{figure}
	
	\begin{figure}[h]
		\centering
		\includegraphics[width=0.54\linewidth]{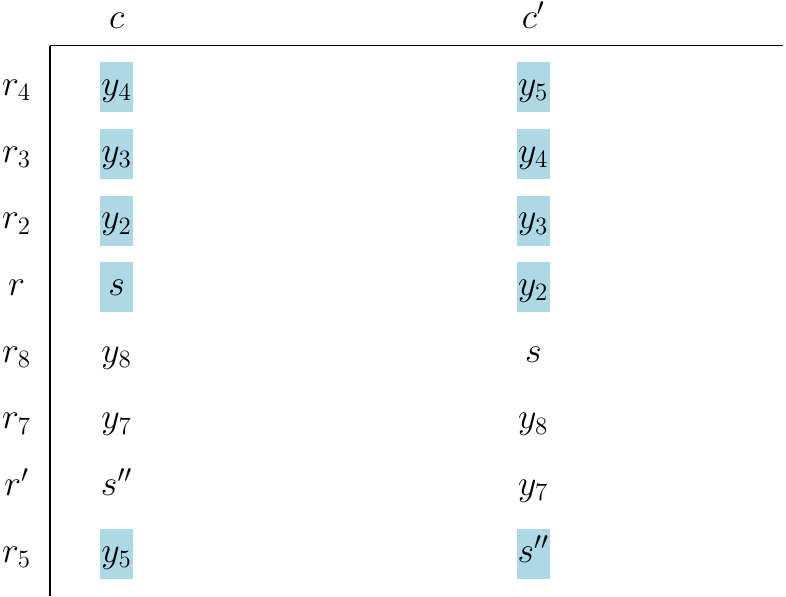}
		\caption{A partial column cycle $\gamma_L(c,c',r\rightarrow r')$ with $\dist_L^\gamma(c,c',r\rightarrow r')=5$.   }
		\label{fig:partial col}
	\end{figure}
	
	We will also need the analogous notion of a partial column cycle. Let $\{c, c'\} \subseteq [n]$ with $c \neq c'$ and let $\pi^\gamma_{c, c'}$ be the column permutation of $L$ defined by these columns.
    Let $r \in [n]$ and let $r' \neq r$ be a row that is hit by the column cycle $\gamma_L(c, c', r)$.
	Let $s = L[r, c]$ and let $s'' = L[r', c]$. Write the cycle of $\pi^\gamma_{c, c'}$ containing $s$ as $(y_1 = s, y_2, \ldots, y_m = s'', y_{m+1}, \ldots, y_\ell)$. For $i \in [m]$, let $r_i$ be the row such that $L[r_i, c] = y_i$. Note that $r = r_1$ and $r' = r_m$. Define $\gamma_L(c, c', r\rightarrow r')$ to be the set of entries
	\[
	\{(r_i, c, y_i), (r_i, c', y_{i+1}) : i \in [m-1]\}.
	\]
	We say that $\gamma_L(c, c', r\rightarrow r')$ is the \emph{partial column cycle} from $r$ to $r'$ along $\gamma_L(c, c', r)$. We define the \emph{distance} from $r$ to $r'$ along $\gamma_L(c, c', r)$, denoted by $\dist_L^\gamma(c, c', r\rightarrow r')$, to be the number $m-1$ of rows hit by $\gamma_L(c, c', r\rightarrow r')$. We extend this distance function to all $(c, c', r, r') \in [n]^4$ with $c \neq c'$ and $r \neq r'$ by setting $\dist_L^\gamma(c, c', r\rightarrow r') = \infty$ if $r'$ is not hit by $\gamma_L(c, c', r)$.

	\subsection{Cross-switching}	 \label{s:cross switch}
	Having defined  partial cycles, we are now ready to define our second switching procedure, known as \textit{cross-switching}. Let $\{c, c', r\} \subseteq [n]$ with $c \neq c'$. Let $r' \neq r$ be a row such that $1 < \dist_L^\gamma(c, c', r\rightarrow r') \neq \infty$ and $\dist_L^\rho(r, r', c\rightarrow c') \neq \infty$. Let $c''$ be the column such that $L[r, c''] = L[r', c]$ and label this symbol $s''$. Define $\zeta_L(c, c', r, r')$ to be the set of entries
	\[
	(\gamma_L(c, c', r\rightarrow r') \setminus \{(r, c, L[r, c])\}) \cup \{(r', c, s'')\} \cup \rho_L(r, r', c''\rightarrow c').
	\]
	
	An example of the set $\zeta_L(c,c',r,r')$ is given in \cref{fig:cross_switch} and we remark that the Latin square on the left is  the same Latin square as in \cref{fig:partial row} with $s''=z_2$, $t=z_6$ and $c''=c_2$ and also the same Latin square as in \cref{fig:partial col} with $s'=y_2$ and $t=y_7$. 
	\begin{figure}[h]
		\centering
		\includegraphics[width=0.84\linewidth]{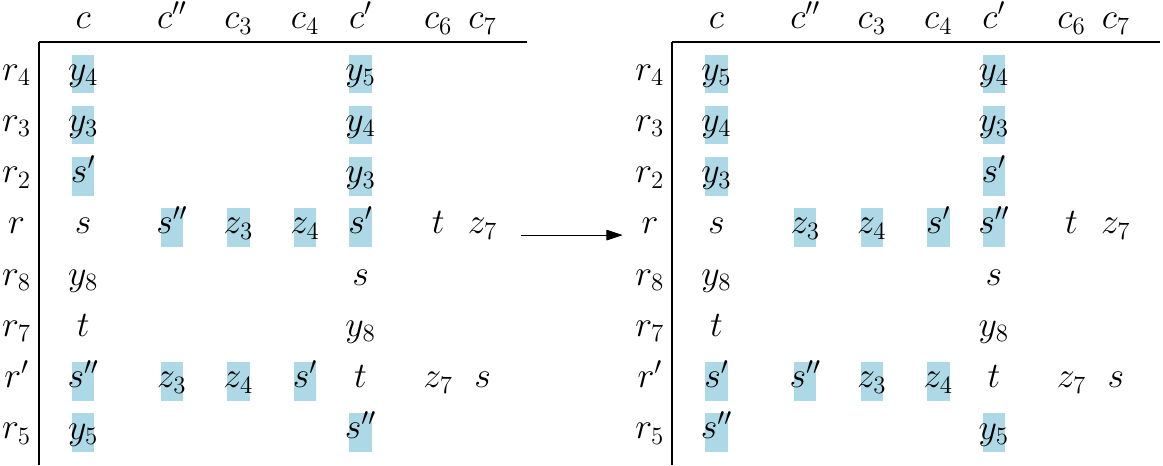}
		\caption{The highlighted entries in the Latin square $L$ on the left form the substructure $\zeta = \zeta_L(c, c', r, r')$. The Latin square on the right is obtained from $L$ by switching on $\zeta$.}
		\label{fig:cross_switch}
	\end{figure}
	
	Let $\cR$ be the set of rows other than $r$ hit by $\gamma_L(c, c', r\rightarrow r')$ and let $\cC$ be the set of columns hit by $\rho_L(r, r', c''\rightarrow c')$. We can define a new Latin square $L'$ by replacing the entries in $\zeta_L(c, c', r, r')$ by the entries in
	\[
	\{(x, c, L[x, c']), (x, c', L[x, c]) : x \in \cR\} \cup   \{(r, y, L[r', y]), (r', y, L[r, y]) : y \in \cC\} \cup  \{(r, c', s''), (r', c, L[r, c'])\}.
	\]
	In other words, for each row in $\cR$, we swap the two symbols involved in $\gamma_L(c, c', r\rightarrow r')$ in that row, for each column in $\cC$, we swap the two symbols involved in $\rho_L(r, r', c'' \rightarrow c')$ in that column, and we swap the symbols in cells $(r, c')$ and $(r', c)$. We say that $L'$ has been obtained from $L$ by switching on $\zeta_L(c, c', r, r')$. See \cref{fig:cross_switch} for an example of cross-switching. Cross-switching has previously been used, for example, by Cavenagh, Greenhill and Wanless~\cite{cycstrucrandom}.
	
	The following lemma, proven in~\cite{cycstrucrandom}, determines the reversibility of  cross-switching.
	
	\begin{lem}[Reversibility for cross-switching]\label{l:cross_reverse}
		Let $L$ be a Latin square of order $n$ and let $\{r, r', c, c'\} \subseteq [n]$ with $r \neq r'$ and $c \neq c'$. 
		If $L$ is obtained from some Latin square $L'$ by switching on $\zeta_{L'}(c, c', r, r')$, then $\zeta_L(c, c', r, r')$ is well-defined and $L'$ is uniquely determined: it is the Latin square obtained from $L$ by switching on $\zeta_L(c, c', r, r')$.
	\end{lem}

	As with cycle switches, performing a cross-switch can affect the cycles that it intersects. The following lemma, also proven by Cavenagh, Greenhill and Wanless \cite{cycstrucrandom}, records how performing a cross-switch changes whether  columns are hit  by certain row cycles in the Latin square. 
	
	\begin{lem}[The effect of a cross-switch on a row cycle]\label{l:cross_switch}
		Let $L$ be a Latin square of order $n$. Let $\{c, c', r\} \subseteq [n]$ with $c \neq c'$. Write the cycle of $\pi^\gamma_{c, c'}$ containing symbol $L[r, c]$ as $(y_1 = L[r, c], y_2, \ldots, y_m)$ and for $i \in [m]$, let $r_i$ be the row such that $L[r_i, c] = y_i$. Let $\{i, j\} \subseteq [m]$ with $i < j-1$ and suppose that $\dist_L^\rho(r_i, r_j, c \rightarrow c') \neq \infty$. Let $L'$ be the Latin square obtained from $L$ by switching on $\zeta_L(c, c', r_i, r_j)$. The set of rows hit by $\gamma_{L'}(c, c', r)$ is exactly the set of rows hit by $\gamma_L(c, c', r)$. 
		Let $\{u, v\} \subseteq [m]$ with $u \neq v$ and $\{u, v\} \cap \{i, j\} = \emptyset$. 
        \begin{enumerate}
            \item \label{item:no_split}
        If $|\{u, v\} \cap [i+1, j-1]| \neq 1$, then $\rho_{L'}(r_u, r_v, c)$ hits column $c'$ if and only if $\rho_L(r_u, r_v, c)$ hits column $c'$. 
         \item  \label{item:split} If $|\{u, v\} \cap [i+1, j-1]| = 1$, then $\rho_{L'}(r_u, r_v, c)$ hits column $c'$ if and only if $\rho_L(r_u, r_v, c)$ does not hit column $c'$.  
            \end{enumerate}
	\end{lem}

\begin{rem} \label{rem:split}
     Let $L'$ be a Latin square obtained from $L$ by switching on $\zeta_L(c,c',r_i,r_j)$ and suppose that we are in situation  \eqref{item:split} and we have $\{u,v\}\subset [m]$ with $|\{u, v\} \cap [i+1, j-1]| = 1$. Now considering $L'$, write the cycle of $\pi^\gamma_{c, c'}$ containing symbol $L'[r, c]=L[r,c]$ as $(y'_1 = L[r, c], y'_2, \ldots, y'_m)$ and for $i \in [m]$, let $r'_i$ be the row such that $L'[r'_i, c] = y'_i$. Then if $r_i=r'_{i'},r_j=r'_{j'},r_u=r'_{u'}$ and $r_v=r'_{v'}$, we have that  $i=i'$, $j=j'$ and $|\{u', v'\} \cap [i+1, j-1]| = 1$  still in the new square $L'$. Indeed the ordering $r'_1,\ldots, r'_{m}$ is obtained from the ordering $r_1,\ldots,r_m$, by inverting the subsequence $r_{i+1},\ldots,r_{j-1}$ and leaving all other entries in the sequence fixed. 
\end{rem}

An example of the situation \eqref{item:split} of Lemma \ref{l:cross_switch} is given in Figure \ref{fig:cross_switch_effect}. The Latin squares in this figure are the same as those in Figure \ref{fig:cross_switch}. Here the cross-switch is on $\zeta_L(c, c', r, r')$ and so $i=1$ and $j=6$. The row switch is on $\rho_L(r_4,r_8,c)$ and so $u=4$ and $v=8$ with $|\{u, v\} \cap [i+1, j-1]| = |\{4, 8\} \cap [2, 5]| =1$. 
    
	\begin{figure}[h]
		\centering
		\includegraphics[width=0.84\linewidth]{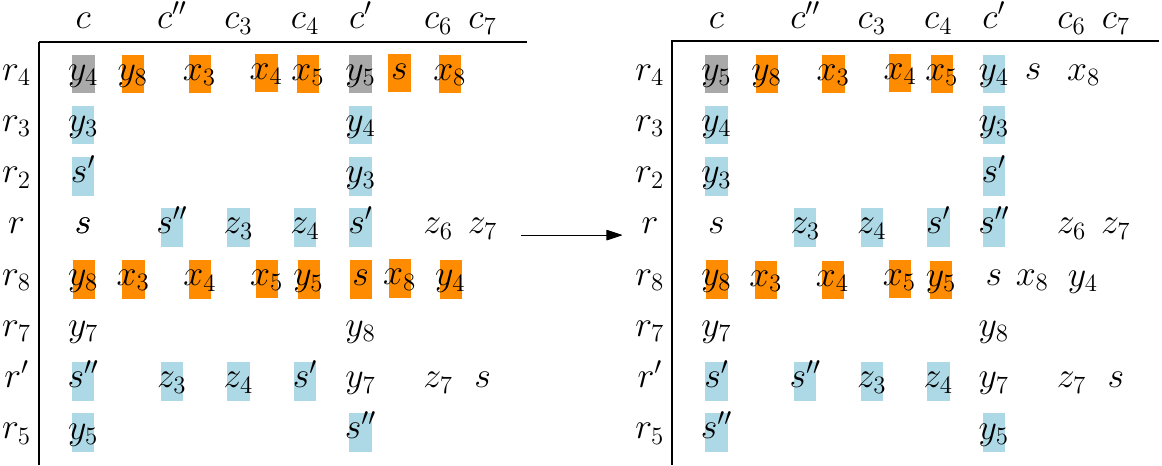}
		\caption{An example of a row cycle $\rho_L(r_4,r_8,c)$ (in orange/grey) before and after a cross-switch on $\zeta_L(c, c', r, r')$ (in blue/grey). }
		\label{fig:cross_switch_effect}
	\end{figure}

	\subsection{$\eta$-switching}\label{ss:eta_switch}
	
	We are now ready to define our  final switching procedure, which we call $\eta$-switching. Let $\{r, r', c\} \subseteq [n]$ with $r \neq r'$. We will define a set $\eta = \eta_L(r, r', c)$ of entries of $L$. Let $c'$ be the column such that $L[r, c'] = L[r', c]$ and let $r''$ be the row such that $L[r'', c'] = L[r, c]$. Also let $s = L[r, c]$ and let $s' = L[r', c]$. The definition of $\eta$ has three different cases and we say that $\eta$ is of Type One, Two, or Three depending on which case it falls into.

\smallskip

	\noindent\textbf{Type One:} $r''=r'$. In this case, $\rho_L(r, r', c)$ is an intercalate and we set 
	\[
	\eta = \rho_L(r, r', c).
	\] 
	
	\noindent\textbf{Type Two:} $r'' \neq r'$ and $\dist_L^\rho(r', r'', c'\rightarrow c) = \infty$. In this case we set 
	\[
	\eta = \{(r, c, s), (r', c, s'), (r, c', s')\} \cup \rho_L(r', r'', c').
	\]
	
	\noindent\textbf{Type Three:} $r'' \neq r'$ and $\dist_L^\rho(r', r'', c'\rightarrow c) \neq \infty$. In this case, let $c''$ be the column such that $L[r', c''] = s$. Set
	\[
	\eta = \{(r, c, s), (r', c, s'), (r, c', s'), (r'', c', s)\} \cup \rho_L(r', r'', c''\rightarrow c).
	\]

\begin{figure}[h]
		\centering
		\includegraphics[width=0.64\linewidth]{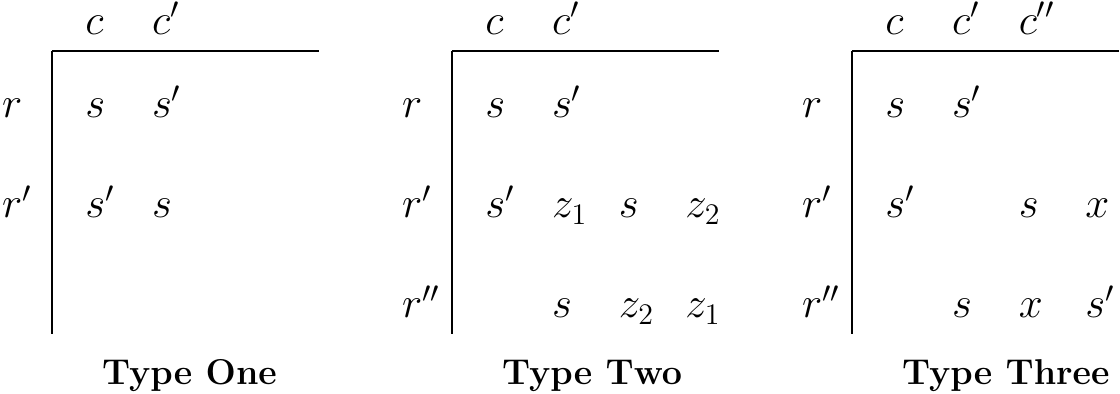}
		\caption{Examples of the three types of $\eta=\eta_L(r,r',c)$.  }
		\label{fig:eta_set}
	\end{figure}
	
	Next, we define how to switch on the substructure $\eta$. 
	
	\smallskip
	
	\noindent\textbf{Type One:} To switch on $\eta$, we simply switch on the row cycle $\rho_L(r, r', c)$. 
	
	\noindent\textbf{Type Two:} To switch on $\eta$, we first switch on $\rho_L(r', r'', c')$ to create the intermediate Latin square $L'$. The substructure $\rho_{L'}(r, r', c)$ is an intercalate and so $\eta_{L'}(r, r', c)$ is of Type One. We then proceed by switching $L'$ on $\rho_{L'}(r, r', c)$.
	
	\noindent\textbf{Type Three:} Let $\cC$ be the set of columns hit by $\rho_L(r', r'', c''\rightarrow c)$. To switch on $\eta$, we replace the entries in $\eta$ by the entries in
	\[
	\{(r, c, s'), (r', c, s), (r, c', s), (r'', c', s')\} \cup \{(r', y, L[r'', y]), (r'', y, L[r', y]) : y \in \cC\}.
	\]
	In other words, for each column hit by $\rho_L(r', r'', c''\rightarrow c)$, we swap the two symbols involved in $\rho_L(r', r'', c''\rightarrow c)$ in that column, as well as swapping $s$ and $s'$ in the cells $(r, c)$, $(r', c)$, $(r, c')$, and $(r'', c')$.
	
	\begin{figure}[h]
		\centering
		\includegraphics[width=0.64\linewidth]{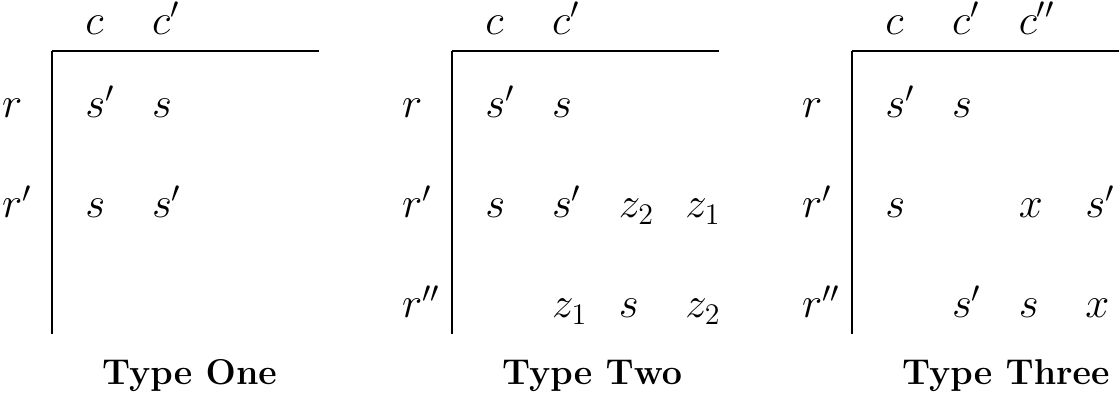}
		\caption{The examples from Figure \ref{fig:eta_set} after performing the $\eta$-switch.  }
		\label{fig:eta_switch}
	\end{figure}
	
	Regardless of whether $\eta$ is of Type One, Two, or Three, when switching $L$ on $\eta$, the entries $(r, c, s)$, $(r', c, s')$, and $(r, c', s')$ are replaced by $(r, c, s')$, $(r', c, s)$, and $(r, c', s)$, respectively. Furthermore, the only other changes to $L$ occur in rows $r'$ and $r''$. We also note that $\eta$-switching is simply a combination of cycle switching and cross-switching. If $\eta$ is of Type One, then switching on $\eta$ just amounts to switching on a row cycle. If $\eta$ is of Type Two, then switching on $\eta$ amounts to switching on two row cycles. If $\eta$ is of Type Three, then switching on $\eta$ is a special case of cross-switching. Explicitly, switching on $\eta$ is equivalent to switching on $\zeta_L(c', c, r', r'')$. In this case, $\dist_L^\gamma(c', c, r'\rightarrow r'') = 2$. 
    See Figure~\ref{f:eta_switch} for a further example of a Type Two $\eta$-switch.  
    This switching procedure has previously been used, for example, in~\cite{genLS,Pitt}.
	
	\begin{figure}[ht]
		\[
		\begin{pmatrix}
			1&2&3&4&5&6&7&8&9&10\\
			4&3&\col5&\col6&\col9&2&\col1&10&\col7&8\\
			10&\col4&\col6&\col9&\col7&8&\col5&2&\col1&3\\
			2&5&1&3&4&7&8&6&10&9\\
			3&6&2&1&8&5&10&9&4&7\\
			6&7&4&2&10&1&9&3&8&5\\
			7&8&9&10&1&3&2&4&5&6\\
			8&\col1&10&5&3&9&\col4&7&6&2\\
			9&10&7&8&2&4&6&5&3&1\\
			5&9&8&7&6&10&3&1&2&4\\
		\end{pmatrix}
		\quad\longrightarrow\quad
		\begin{pmatrix}
			1&2&3&4&5&6&7&8&9&10\\
			4&3&\col6&\col9&\col7&2&\col5&10&\col1&8\\
			10&\col1&\col5&\col6&\col9&8&\col4&2&\col7&3\\
			2&5&1&3&4&7&8&6&10&9\\
			3&6&2&1&8&5&10&9&4&7\\
			6&7&4&2&10&1&9&3&8&5\\
			7&8&9&10&1&3&2&4&5&6\\
			8&\col4&10&5&3&9&\col1&7&6&2\\
			9&10&7&8&2&4&6&5&3&1\\
			5&9&8&7&6&10&3&1&2&4\\
		\end{pmatrix}
		\]
		\caption{\label{f:eta_switch}The highlighted entries in the Latin square $L$ on the left form the substructure $\eta = \eta_L(8, 3, 2)$, which is of Type Two. The Latin square on the right is obtained from $L$ by switching on $\eta$.}
	\end{figure}
	
	Note that $\eta$-switching does not immediately generalise to $k \times n$ Latin rectangles with $k < n$, since the row $r''$ may not exist in this case. This is the primary reason why the results of this paper apply to Latin squares only.

We now determine the reversibility of $\eta$-switching.

\begin{lem}[Reversibility for $\eta$-switching]\label{l:eta_reverse}
	Let $L$ be a Latin square of order $n$ and let $\{r, r', c\} \subseteq [n]$ with $r \neq r'$. If $\eta_L(r, r', c)$ is of Type One, then there are at most $n$ Latin squares $L'$ from which $L$ can be obtained by switching on $\eta_{L'}(r, r', c)$. If $\eta_L(r, r', c)$ is not of Type One, then there is at most one Latin square $L'$ from which $L$ can be obtained by switching on $\eta_{L'}(r, r', c)$.
\end{lem}
\begin{proof}
	Suppose that $L'$ is a Latin square from which $L$ can be obtained by switching on $\eta_{L'}(r, r', c)$. Let $\eta_1 = \eta_L(r, r', c)$ and let $\eta_2 = \eta_{L'}(r, r', c)$. The following facts are easy to verify:
	\begin{itemize}
		\item If $\eta_2$ is of Type One, then so is $\eta_1$.
		\item If $\eta_2$ is of Type Two, then $\eta_1$ is of Type One.
		\item If $\eta_2$ is of Type Three, then so is $\eta_1$.
	\end{itemize}
	In particular, if $\eta_1$ is of Type Two, then no such $L'$ can exist.
	
	Suppose that $\eta_1$ is of Type One.	Let $c'$ be the column of $L$ such that $L[r, c'] = L[r', c]$. If $\eta_2$ is of Type One also, then there is exactly one choice for $L'$, namely, the square obtained from $L$ by switching on $\eta_1$. If $\eta_2$ is of Type Two, then $L'$ can be obtained from $L$ by first switching on $\eta_1$ to produce a Latin square $L''$, then switching $L''$ on a row cycle $\rho_{L''}(r', r'', c')$ for some row $r''$ such that $\rho_{L''}(r', r'', c')$ does not hit column $c$. There are trivially at most $n-1$ choices for $r''$, thus there are at most $n-1$ choices for $L'$ in this case. Therefore, the total number of choices for $L'$ is at most $n$.
	
	Now suppose that $\eta_1$ is of Type Three. Let $r_1$ be the row such that $L[r_1, c'] = L[r, c]$ and let $r_2$ be the row such that $L'[r_2, c'] = L'[r, c]$. It is easy to verify using the definition of $\eta$-switching that $r_1 = r_2$. Thus $L$ has been obtained from $L'$ by switching on $\eta_{L'}(r, r', c) = \zeta_{L'}(c', c, r', r_1)$. Thus, by \cref{l:cross_reverse}, $L'$ is uniquely determined: it is the Latin square obtained from $L$ by switching on $\zeta_L(c', c, r', r_1)$. This completes the proof.
\end{proof}

\section{Intercalates}\label{s:interc}

Our general strategy to prove \cref{t:main} is as follows. Let $\bL$ be a random Latin square of order $n$ and let $P = \{q_1, q_2, \ldots, q_N\}$ be a partial Latin square of order $n$ satisfying the hypotheses of \cref{t:main}. Let $P_0 = \emptyset$ and for $i \in [N]$, let $P_i = \{q_j : 1 \leq j \leq i\}$. By the chain rule of probability,
\[
\Pr(\bL \supseteq P) = \prod_{i=1}^{N} \Pr(\bL \supseteq P_i | \bL \supseteq P_{i-1}).
\]
We will use $\eta$-switching to give upper and lower bounds for $\Pr(\bL \supseteq P_i | \bL \supseteq P_{i-1})$ for each $i \in [N]$. By \cref{l:eta_reverse}, in order to do this effectively, we must be able to upper bound the probability that certain substructures $\eta_{\bL}(r, r', c)$ are of Type One, when we condition on the event that $\bL$ contains $P_{i-1}$. Equivalently, we need to be able to upper bound the probability that $\bL$ has an intercalate involving rows $r$, $r'$ and column $c$, given that $\bL$ contains $P_{i-1}$. 

Intercalates in random Latin squares is a well studied topic~\cite{cycstrucrandom, LSparity, KSS, KSSS, KS, manysubsq}.  If $\bL$ is a random Latin square of order $n$ and $\{r, r', c\} \subseteq [n]$ with $r \neq r'$, then the probability that $\bL$ has an intercalate involving rows $r$, $r'$ and column $c$ is $(1+o(1))/n$. It is natural to believe that if we condition on the event that $\bL$ contains a particular substructure that is not `too large', then the probability of this event should not change by much. However, none of the previous work on intercalates in random Latin squares allows for such a conditioning.  
\cref{t:interc} below, which is the main result of this section, partially deals with this situation by giving an upper bound on the probability that $\bL$ has an intercalate involving rows $r$, $r'$ and column $c$, given that $\bL$ contains some partial Latin square satisfying the hypotheses of \cref{t:main}.  \cref{t:interc} may be of independent interest. To prove it, we will employ all the switching procedures discussed in~\cref{s:switch}. 

\begin{thm}\label{t:interc}
	Let $P$ be a partial Latin square of order $n$. Let $\alpha, \beta > 0$ be such that $|\cR_P| \leq \alpha n$ and $|\cC_P| \leq \beta n$ and suppose that $2\alpha+\beta<1$. Let $\bL$ be a random Latin square of order $n$ conditioned on the event that $\bL$ contains $P$.
	Let $\{r, r', c\} \subseteq [n]$ with $r' \not\in \cR_P$. The probability that $\eta_{\bL}(r, r', c)$ is of Type One is at most $D/n$ for some constant $D = D(\alpha, \beta)$ that satisfies $D \to 22$ as $\alpha, \beta \to 0$.
\end{thm}
\begin{proof}
	Let $Y_0$ be the set of Latin squares of order $n$ that contain $P$. For $L \in Y_0$, let $r_0 = r_0(L) = r$, let $c_0 = c_0(L) = c$, and let $s_0 = s_0(L) = L[r_0, c_0]$. Also let $r_1 = r_1(L) = r'$ and let $c_1 = c_1(L)$ be the column such that $L[r_0, c_1] = L[r_1, c_0]$. For $1\leq j < n(1-2\alpha-\beta)/3$,  we make the following recursive definitions:   
	\begin{itemize}
		\item For $L \in Y_{j-1}$, let $s_j(L) = L[r_j(L), c_j(L)]$.
		\item For $L \in Y_{j-1}$, let $r_{j+1} = r_{j+1}(L)$ be the row such that $L[r_{j+1}, c_j(L)] = s_{j-1}(L)$. 
		\item For $L \in Y_{j-1}$, let $c_{j+1} = c_{j+1}(L)$ be the column such that $L[r_{j}(L), c_{j+1}] = s_{j-1}(L)$. 
		\item Let $X_j$ be the set of squares $L \in Y_{j-1}$ such that $s_j(L) = s_{j-1}(L)$.
		\item Let $Y_j$ be the set of squares $L \in Y_{j-1}$ such that $s_j(L) \neq s_{j-1}(L)$, $r_{j+1}(L) \notin \cR_P \cup \{r_\ell(L) : 0 \leq \ell \leq j\}$, and $c_{j+1}(L) \not\in \cC_P \cup \{c_\ell(L) : 0 \leq \ell \leq j\}$.
		\item Let $p_{j-1} = \Pr(\bL \in X_j | \bL \in Y_{j-1})$.
	\end{itemize}
An example is given in Figure \ref{fig:intercalate chase}.	Note that $p_0 = |X_1|/|Y_0|$ is the probability that we want to bound.

   	\begin{figure}[h]
		\centering
		\includegraphics[width=0.24\linewidth]{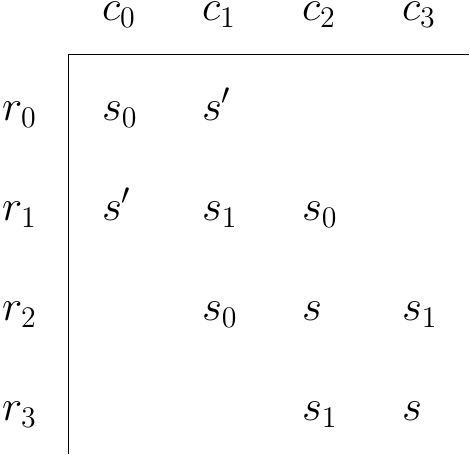}
		\caption{An example of a Latin square $L\in X_3\subseteq Y_2$. Here we have $r_0=r$, $r_1=r'$, $c_0=c$ and $s_2=s_3=s$.  }
		\label{fig:intercalate chase}
	\end{figure} 
	
	The first claim in our proof is the following.
	
	\begin{claim}\label{cl:pjrecurrence}
		For all $1\leq j < n(1-2\alpha-\beta)/3$,
		\begin{equation*}\label{e:pjrecurrence}
			p_{j-1}	\leq \frac{np_j+1}{n(1+p_j-2\alpha-\beta)-3j-2}.
		\end{equation*}
	\end{claim}
	\begin{proofclaim}
		We prove the claim using $\eta$-switching. We consider switching squares $L \in X_j$ on substructures of the form $\eta_L(r_j(L), x, c_j(L))$ to obtain squares in $Y_j$.  
		
		Let $j \geq 1$. Let $L \in X_j$, let $r_j = r_j(L)$, let $c_j = c_j(L)$, and let $s_{j-1} = s_{j-1}(L) = s_{j}(L)$. Let $x \in [n] \setminus (\cR_P \cup \{r_\ell(L) : 0 \leq \ell \leq j\})$. Let $y$ be the column such that $L[r_j, y] = L[x, c_j]$ and let $x'$ be the row such that $L[x', y] = s_{j-1}$. Suppose that $y \in [n] \setminus (\cC_P \cup \{c_\ell(L) : 0 \leq \ell \leq j\})$ and $x' \in [n] \setminus (\cR_P \cup \{r_\ell(L) : 0 \leq \ell \leq j\})$ and let $L'$ be the square obtained from $L$ by switching on $\eta_L(r_j, x, c_j)$. Then $\eta_L(r_j, x, c_j)$ does not hit any non-empty cell of $P$ nor any cell in
		\[
		\{(r_\ell, c_\ell), (r_{\ell+1}, c_\ell), (r_\ell, c_{\ell+1}) : 0 \leq \ell \leq j-1\}.
		\]
		Therefore, $L' \in Y_{j-1}$. Furthermore, $s_j(L') = L[x,c_j] \neq s_{j-1} = s_{j-1}(L')$. Also, $L'[x, c_j] = s_{j-1}(L')$ and $L'[r_j, y] = s_{j-1}(L')$. Therefore, $r_{j+1}(L') = x \in [n] \setminus (\cR_P \cup \{r_\ell(L) : 0 \leq \ell \leq j\})$ and $c_{j+1}(L') = y \in [n] \setminus (\cC_P \cup \{c_\ell(L) : 0 \leq \ell \leq j\})$. Hence, $L' \in Y_j$.

        	\begin{figure}[h]
		\centering
		\includegraphics[width=0.64\linewidth]{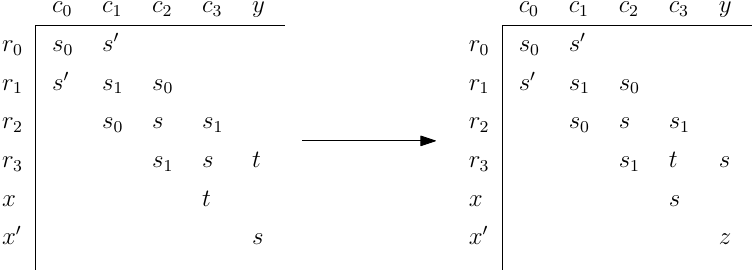}
		\caption{An example of a switch from the Latin square $L\in X_3\subseteq Y_2$ from Figure \ref{fig:intercalate chase}. Here it could be that $z=s$ or not depending on whether the $\eta$-switch was case 3 or case 2, respectively.  }
		\label{fig:intercalate switch}
	\end{figure} 
		
		There are at most $|\cC_P|+j+1$ choices of $x \in [n] \setminus (\cR_P \cup \{r_\ell(L) : 0 \leq \ell \leq j\})$ such that $y \in \cC_P \cup \{c_\ell(L) : 0 \leq \ell \leq j\}$ and there are at most $|\cR_P|+j+1$ choices such that $x' \in \cR_P \cup \{r_\ell(L) : 0 \leq \ell \leq j\}$. Hence, there are at least $n-2|\cR_P|-|\cC_P|-3j-3$ choices of $x \in [n] \setminus (\cR_P \cup \{r_\ell(L) : 0 \leq \ell \leq j\})$ such that switching $L$ on $\eta_L(r_j, x, c_j)$ yields a square in $Y_j$. Thus, the total number of switches from $X_j$ to $Y_j$ is at least
		\begin{equation}\label{e:XjLB}
			|X_j|(n-2|\cR_P|-|\cC_P|-3j-3) \geq |X_j|(n(1-2\alpha-\beta)-3j-3).
		\end{equation}
		
		Now let $L' \in Y_j$, let $r_j = r_j(L')$, let $c_j = c_j(L')$, and let $r_{j+1} = r_{j+1}(L')$. If $L'$ can be obtained from some $L \in X_j$ by switching on $\eta_L(r_j(L), r'', c_j(L))$ for some $r'' \in [n] \setminus (\cR_P \cup \{r_\ell(L) : 0 \leq \ell \leq j\})$, then $r_j(L) = r_j$, $c_j(L) = c_j$, and $r'' = r_{j+1}$. Let $\eta = \eta_{L'}(r_j, r_{j+1}, c_j)$. By \cref{l:eta_reverse}, there are at most $n$ possibilities for the Latin square $L$ if $\eta$ is of Type One and there is at most one possibility for $L$ otherwise. The set of squares $L'' \in Y_j$ such that $\eta_{L''}(r_j(L''), r_{j+1}(L''), c_j(L''))$ is of Type One is exactly $X_{j+1}$. 
		Therefore, the total number of switches from squares in $X_j$ to those in $Y_j$ is at most
		\begin{equation}\label{e:YjUB}
			n|X_{j+1}|+|Y_j \setminus X_{j+1}| \leq |Y_j|(np_j+1), 
		\end{equation}
        and combining this with \cref{e:XjLB}, we get that 
        \[\frac{|Y_j|}{|X_j|}\geq \frac{n(1-2\alpha-\beta)-3j-3}{np_j+1}.\]
		Finally then, we have
		\[
		p_{j-1} = \frac{|X_j|}{|Y_{j-1}|} \leq \frac{|X_j|}{|X_j|+|Y_j|} \leq \frac1{1+\frac{n(1-2\alpha-\beta)-3j-3}{np_j+1}} = \frac{np_j+1}{n(1+p_j-2\alpha-\beta)-3j-2},
		\]
		proving the claim.
	\end{proofclaim}
	
	Suppose that, for some $j \geq 1$, we have an upper bound $p_j \leq g(n, P)$ for some function $g$. We may want to use this in conjunction with \cref{cl:pjrecurrence} to conclude that 
	\[
	p_{j-1} \leq \frac{ng(n, P)+1}{n(1+g(n, P)-2\alpha-\beta)-3j-2}.
	\]
	We could then iterate this process to obtain an upper bound on $p_0$. In order to do this, we need to know that the function $\varphi : [0,1] \to \mathbb{R}$ defined by
	\[
	\varphi(x) = \frac{nx+1}{n(1+x-2\alpha-\beta)-3j-2}
	\]
	is non-decreasing in $x$ for sufficiently large $n$. Let $K = n(1-2\alpha-\beta)-3j-2$. The derivative $\varphi'$ of $\varphi$ is defined by
	\[
	\varphi'(x) = \frac{n(K-1)}{(nx+K)^2}.
	\]
	Therefore, $\varphi$ is non-decreasing in $x$ whenever $K \geq 1$. In particular, $\varphi$ is non-decreasing in $x$ if $j = o(n)$ and $n$ is sufficiently large.
	A tempting way to proceed is to use the described strategy to obtain an upper bound on $p_0$ with some $j = o(n)$ and the trivial upper bound $p_j \leq 1$. It turns out that the best possible upper bound that can be proved using this approach is $p_0 = O(n^{-1/2})$, which is achieved by taking $j = O(n^{1/2})$. Thus, other ideas are needed in combination with \cref{cl:pjrecurrence} to prove the theorem. 
    
    We will prove that $p_2 = O(n^{-1})$. We will then combine this with \cref{cl:pjrecurrence} to conclude that $p_0 = O(n^{-1})$. We prove that $p_2=O(n^{-1})$ in \cref{cl:p2UB}. In that proof, we use a combination of row cycle switching and column cycle switching. The next two claims we need in order to prove a lower bound on the number of switches that we can apply in the proof of \cref{cl:p2UB}.
	
	For a Latin square $L \in Y_2$, let $Q(L)$ denote the set of non-empty cells of $P$ and the cells in
	\[
	\{(r_\ell(L), c_\ell(L)), (r_{\ell+1}(L), c_\ell(L)), (r_\ell(L), c_{\ell+1}(L)) : 0 \leq \ell \leq 2\} \cup\{(r_3,c_3)\}.
	\]
     See Figure \ref{fig:intercalate chase} for the set $Q(L)$ in a square $L$ in which $P=\{(r,c,s_0)\}=\{(r_0,c_0,s_0)\}$. 
	Also let $\gamma_L = \gamma_L(c_2(L), c_3(L), r_1(L))$. Let $G_L$ be the set of rows in $[n] \setminus (\cR_P \cup \{r_0(L), r_1(L)\})$ that are hit by $\gamma_L$ and let $\nu = \nu(L)$ be the integer such that $|G_L|=n-\nu$. Note that if $L \in X_3$, then $G_L \cap \{r_2(L), r_3(L)\} = \emptyset$, since $\gamma_L(c_2(L), c_3(L), r_2(L))$ forms an intercalate with rows $r_2(L)$, $r_3(L)$ and column $c_2(L)$, $c_3(L)$. Therefore, $\nu \in [|\cR_P|+N, n]$, where
    \[
    N = \begin{cases}
        3 & \text{if } r \in \cR_P,\\
        4 & \text{otherwise}.
    \end{cases}
    \]
    Here, the $N$ is accounting for $r_1(L), r_2(L), r_3(L)\notin \cR_P$ and possibly $r_0(L) \notin \cR_P$.
    For $k \in [|\cR_P|+N, n]$, let $A_k$ be the set of squares $L' \in X_3$ with $\nu(L')=k$ and let $a_k = |A_k|$. 
	
	Let $k \in [|\cR_P|+N, n]$ and let $L \in A_k$. For $i \in \{0, 1, 2, 3\}$, let $r_i=r_i(L)$, let $c_i=c_i(L)$, and let $s_i=s_i(L)$. Write the cycle of $\pi_{c_2, c_3}^\gamma$ containing $s_0 = L[r_1, c_2]$ as $(z_1'=s_0, z_2', \ldots, z_m')$. For $i \in [m]$, let $x_i' = x_i'(L)$ be the row such that $L[x_i', c_2] = z_i'$ (and so $x_1'=r_1$). For $i \in [n-k]$, let $x_i = x_i(L)$ be the $i$-th element of the sequence $(x_1', x_2', \ldots, x_m')$ that is not a member of $\cR_P \cup \{r_0(L), r_1(L)\}$. 
	
	Let $T = T_k$ be the set of pairs $(i, j) \in [n-k]^2$ with $i < j$. For $(i, j) \in T$ and $L \in A_k$, we call the triple $(L, i, j)$ \emph{good} if $\rho_L(x_i(L), x_j(L), c_3(L))$ does not hit column $c_2(L)$, otherwise we call the triple \emph{bad}. Let $A_k^{\textnormal{g}}$ be the set of good triples $(L, i, j) \in A_k \times T$ and let $A_k^{\textnormal{b}}$ be the set of bad triples $(L, i, j) \in A_k \times T$. In our next claim, we lower bound the number of good triples in $A_k \times T$.
	
	\begin{claim}\label{cl:Akg}
		For $k \in [|\cR_P|+N, n]$,
		\[
		|A_k^{\textnormal{g}}| \geq \frac{a_k(n-k)(n-k-N)}{9}.
		\]
	\end{claim}
	\begin{proofclaim}
        The claim is trivial if $k \in [n-N,n]$, so we will assume for this proof that $k \in [|\cR_P|+N, n-N-1]$.
		Let $T = T_k$ and note that $T \neq \emptyset$ since $n-k > 2$. For $(i, j) \in T$, define $W_k(i, j)$ to be the set of pairs $(u, v) \in T$ such that $|\{u, v\} \cap [i+1, j-1]|=1$ and $\{u, v\} \cap \{i, j\} = \emptyset$. Let $w_k(i, j) = |W_k(i, j)|$ and note that
		\begin{equation}\label{e:w}
		w_k(i, j) = (j-i-1)(n-k-j+i-1).
		\end{equation}
		Let $w_{\textnormal{max}}$ be the maximum value of $w_k(i, j)$ over all $(i, j) \in T$ and note that
		\begin{equation}\label{e:w_max}
			w_{\textnormal{max}} = \left\lfloor \frac{(n-k-2)^2}4\right\rfloor,
		\end{equation}
	achieved when $j-i \in \{\lfloor (n-k)/2 \rfloor, \lceil (n-k)/2 \rceil\}$.
	
	We will prove the claim using cross-switching. We consider switching from triples in $A_k^{\textnormal{b}}$ to triples in $A_k^{\textnormal{g}}$.
	So let $(L, i_0, j_0) \in A_k^{\textnormal{b}}$. For each $(u, v) \in W_k(i_0, j_0)$, we associate a triple $\textnormal{sw}(L, i_0, j_0, u, v) \in A_k^{\textnormal{g}}$ as follows. If $(L, u, v) \in A_k^{\textnormal{g}}$, then set $\textnormal{sw}(L, i_0, j_0, u, v) = (L, u, v)$. In this case, we say that $(L, u, v)$ has been obtained from $(L, i, j, u, v)$ by a switch of Kind $1$. Now assume that $(L, u, v) \in A_k^{\textnormal{b}}$. For  $i \in \{0, 1, 2, 3\}$, let $r_i = r_i(L)$ and let $c_i = c_i(L)$. For $i \in [n-k]$, let $x_i = x_i(L)$. Let $m$ be the length of $\gamma_L$ and for $i \in [m]$, let $x_i' = x_i'(L)$. Let $L'$ be the square obtained from $L$ by cross-switching on $\zeta = \zeta_L(c_2, c_3, x_u, x_v)$. Note that $\zeta$ is well-defined since $(L, u, v)$ is a bad triple. We next claim that the cells of the entries in $\zeta$ do not intersect $Q(L)$. The only entries in $\zeta$ that do not occur in column $c_2$ or $c_3$ lie in row $x_u$ or $x_v$, both of which are members of $G_L$. 
    By definition, $G_L \cap (\cR_P \cup \{r_0, r_1\}) = \emptyset$ and, as observed before this claim, $G_L \cap \{r_2, r_3\} = \emptyset$. Hence, no cell of an entry in $\zeta$ that is not in column $c_2$ or $c_3$ is a member of $Q(L)$. The only cells in $Q(L)$ with column $c_2$ or $c_3$ are those in $\{(r_1, c_2), (r_2, c_2), (r_2, c_3), (r_3, c_2), (r_3, c_3)\}$. Since $\gamma_L$ does not hit rows $r_2$ or $r_3$, it follows that no entry in $\zeta$ lies in any cell in $\{(r_2, c_2), (r_2, c_3), (r_3, c_2), (r_3, c_3)\}$. 
    Finally, note that the set of rows hit by $\zeta$ is of the form $\{x_\ell' : u_0 \leq \ell \leq v_0\}$ for some $(u_0, v_0) \in T_m$ where $x_{u_0}' = x_u$ and $x_{v_0}' = x_v$. Since $x_1' = r_1 \notin G_L$, it follows that $u_0 > 1$ and so $\zeta$ does not hit cell $(r_1, c_2)$. We have thus proven that the cells of the entries in $\zeta$ do not intersect $Q(L)$. 
    Hence, $L' \in X_3$. Furthermore, \cref{l:cross_switch} implies that $L' \in A_k$. Let $(i', j') \in [n-k]^2$ be such that $x_{i'}(L') = x_{i_0}(L)$ and $x_{j'}(L') = x_{j_0}(L)$. 
    Note that since $i_0 < j_0$ and $|\{u, v\} \cap [i_0+1, j_0-1]|=1$, it follows   from Remark \ref{rem:split}  that $i' < j'$ and $(u, v) \in W_k(i', j')$ also. 
    By \cref{l:cross_switch} \eqref{item:split}, $(L', i', j') \in A_k^{\textnormal{g}}$ and so we set $\textnormal{sw}(L, i_0, j_0, u, v)=(L', i', j')$. In this case, we say that $(L', i', j')$ has been obtained from $(L, i_0, j_0, u, v)$ by a switch of Kind $2$.
	
	Now let $(L', i', j') \in A_k^{\textnormal{g}}$. We want to determine the number of quintuples $(L, i, j, u, v)$ where $(L, i, j) \in A_k^{\textnormal{b}}$, $(u, v) \in W_k(i, j)$, and $(L', i', j') = \textnormal{sw}(L, i, j, u, v)$. First suppose that $(L', i', j')$ has been obtained from such a $(L, i, j, u, v)$ by a switch of Kind $1$. Then $L = L'$, $(u, v) = (i', j')$, and $(i, j) \in T$ is such that $(i', j') \in W_k(i, j)$, which is equivalent to $(i, j) \in W_k(i', j')$. 
    Thus, there are at most $w_k(i', j')$ choices for $(L, i, j, u, v)$ in this case. 
    Now suppose that $(L', i', j')$ has been obtained from some $(L, i, j, u, v)$ by a switch of Kind $2$. As previously observed, by Remark \ref{rem:split}, we must have $(u,v)\in W_k(i',j')$. 
 Fix some $(u, v) \in W_k(i', j')$ and suppose that $(L, i, j) \in A_k^{\textnormal{b}}$  is such that $(L', i', j') = \textnormal{sw}(L, i, j, u, v)$. By \cref{l:cross_reverse} and Remark \ref{rem:split}, we know that $L$ is the square obtained from $L'$ by switching on $\zeta_{L'}(c_2(L), c_3(L), x_{u}(L), x_{v}(L))$. Moreover, the positions $i,j$ can be recovered by considering the positions of the rows $x_{i'}(L')$ and $x_{j'}(L')$ after performing the cross-switch. That is, the entire quintuple $(L, i, j, u, v)$ is uniquely determined from the pair $(u, v) \in W_k(i', j')$. It follows that the total number of choices for $(L, i, j, u, v)$ in this case is at most $w_k(i', j')$. Hence, the total number of such quintuples in general is at most $2w_k(i', j')$. 
	
	By double counting, we obtain the inequality
	\begin{equation}\label{e:AkgvsAkb}
	\sum_{(L, i, j) \in A_k^{\textnormal{b}}} w_k(i, j) \leq 2\sum_{(L', i', j') \in A_k^{\textnormal{g}}} w_k(i', j'), 
	\end{equation}
    and so 
\begin{equation}\label{e:AkgvsAkb2}
	\sum_{(L, i, j) \in A_k^{\textnormal{b}}} w_k(i, j) +\sum_{(L', i', j') \in A_k^{\textnormal{g}}} w_k(i', j') \leq 3\sum_{(L', i', j') \in A_k^{\textnormal{g}}} w_k(i', j'). 
	\end{equation}

	We also know that
	\begin{align*}
		\sum_{(L, i, j) \in A_k^{\textnormal{b}}} w_k(i, j) + \sum_{(L', i', j') \in A_k^{\textnormal{g}}} w_k(i', j') &= \sum_{L \in A_k} \sum_{(i, j) \in T} w_k(i, j)\\
		&= a_k\sum_{(i, j) \in T} (j-i-1)(n-k-j+i-1)\\
		&= \frac{a_k(n-k)(n-k-1)(n-k-2)(n-k-3)}{12},
	\end{align*}
where we have used \cref{e:w}. Combining with \cref{e:AkgvsAkb2}, we obtain
\begin{equation*}\label{e:Akg_prelim}
\frac{a_k(n-k)(n-k-1)(n-k-2)(n-k-3)}{36} \leq \sum_{(L', i', j') \in A_k^{\textnormal{g}}} w_k(i', j') \leq |A_k^\textnormal{g}|w_{\textnormal{max}} \leq |A_k^\textnormal{g}|\frac{(n-k-2)^2}4,
\end{equation*}
from \cref{e:w_max}. Therefore,
\[
|A_k^{\textnormal{g}}| \geq \frac{a_k(n-k)(n-k-1)(n-k-3)}{9(n-k-2)} \geq \frac{a_k(n-k)(n-k-N)}{9}, 
\]
as required.
	\end{proofclaim}

Next, define
\begin{equation}\label{e:Sigma_1}
	\Sigma_1 = \sum_{k=|\cR_P|+N}^{n} (k-|\cR_P|-N)a_k.
\end{equation}
Our next claim gives a lower bound on $\Sigma_1$.

\begin{claim}\label{cl:Sigma_1}
	\[
	\Sigma_1 \geq \frac{|X_3|n(1-\alpha)}{10}(1-o(1)).
	\]
\end{claim}
\begin{proofclaim}
	We prove the claim by double counting the number of ways there are to switch a Latin square $L \in X_3$ on a row cycle of the form $\rho_L(x_i(L), x_j(L), c_3(L))$ to obtain a square $L' \in X_3$ for which $\nu(L') > \nu(L)$.
	
	Let $L \in A_k$ for some $k \in [|\cR_P|+N, n]$, let $T = T_k$, and let $(i_0, j_0) \in T$ be such that $(L, i_0, j_0) \in A_k^{\textnormal{g}}$. Then by definition of $A_k^{\textnormal{g}}$ and \cref{l:row_cyc_switch} \eqref{i:col both intersect}, the Latin square obtained from $L$ by switching on $\rho_L(x_{i_0}(L), x_{j_0}(L), c_3(L))$ is a member of $A_{k'}$ where $k' = k+j-i > k$. So the number of switches from some square $L \in A_k$ to some $L' \in X_3$ with $\nu(L') > \nu(L)$ is at least $|A_k^g|$. Thus, the total number of switches from some square $L \in X_3$ to some $L' \in X_3$ with $\nu(L') > \nu(L)$ is at least
	\begin{equation}\label{e:LB}
		\sum_{k=|\cR_P|+N}^{n}|A_k^{\textnormal{g}}| \geq \sum_{k=|\cR_P|+N}^{n} \left(\frac{a_k(n-k)(n-k-N)}{9}\right),
	\end{equation}
by \cref{cl:Akg}.

	Now let $L' \in A_{k'}$ for some $k' \in [|\cR_P|+N, n]$ and suppose that $L'$ has been obtained from some $L \in A_k$ for some $k < k'$ by switching on $\rho = \rho_L(x_{i_0}(L), x_{j_0}(L), c_3(L))$ for some $(i_0, j_0) \in T_k$. For  $i \in [n-k]$, let $x_i = x_i(L)$ and for $i \in \{0, 1, 2, 3\}$, let $r_i = r_i(L)$ and let $c_i = c_i(L)$.
    Observe that $\{x_{i_0}, x_{j_0}\} \cap (\cR_P \cup \{r_0, r_1, r_2, r_3\}) = \emptyset$ and so no entry of $\rho$ has a cell in $Q(L)$. Thus, $c_3(L') = c_3$. Also observe from \cref{f:row_cyc_reverse} that $L$ is equal to the square obtained from $L'$ by switching on $\rho_{L'}(x_{i_0}, x_{j_0}, c_3(L'))$. Hence, the pair $(x_{i_0}, x_{j_0})$ uniquely determines $L$. From \cref{l:row_cyc_switch} \eqref{i:col one intersect}, we know that exactly one of $x_{i_0}$ and $x_{j_0}$ is an element of $G_{L'}$ and exactly one is an element of $[n] \setminus (\cR_P \cup G_{L'} \cup \{r_0(L'), r_1(L'), r_2(L'), r_3(L')\})$. It follows that there are at most $(n-k')(k'-|\cR_P|-N)$ choices for $(x_{i_0}, x_{j_0})$ and the same upper bound applies for the number of choices of $L$. Hence, the number of switches from some square $L \in X_3$ to some $L' \in X_3$ with $\nu(L') > \nu(L)$ is at most
	\begin{equation}\label{e:UB}
		\sum_{k'=|\cR_P|+N}^{n} a_{k'}(n-k')(k'-|\cR_P|-N).
	\end{equation}

	Combining \cref{e:LB} and \cref{e:UB}, we obtain
	\begin{align}\label{e:sumLB}
	0 &\leq \sum_{k=|\cR_P|+N}^{n} a_k(n-k)\left(9(k-|\cR_P|-N)-(n-k-N)\right)\nonumber \\
	&= \sum_{k=|\cR_P|+N}^{n} a_k \left((n-|\cR_P|-N)-(k-|\cR_P|-N)\right)\left(10(k-|\cR_P|-N)-(n-|\cR_P|-N)+N\right) \nonumber \\ 
    &=(11(n-|\cR_P|-N)-N)\Sigma_1-10\Sigma_2-(n-|\cR_P|-N)(n-|\cR_P|-2N)|X_3|,
	\end{align}
	where
	\[
	\Sigma_2=\sum_{k=|\cR_P|+N}^{n} (k-|\cR_P|-N)^2a_k.
	\]
	Cauchy-Schwarz thus implies that
	\begin{equation}\label{e:Sigma_2LB}
	\Sigma_2 \geq \frac{\Sigma_1^2}{|X_3|}.
	\end{equation}
	Multiplying \cref{e:sumLB} by $|X_3|$ and applying  \cref{e:Sigma_2LB} gives the inequality
	\begin{align*}
		0&\geq 10\Sigma_1^2-(11(n-|\cR_P|-N)-N)|X_3|\Sigma_1+(n-|\cR_P|-N)(n-|\cR_P|-2N)|X_3|^2
        \\ &=\left(\Sigma_1-(n-|\cR_P|-N)|X_3|\right)\left(10\Sigma_1-(n-|\cR_P|-2N)|X_3|\right).
	\end{align*}
	As the first factor is negative (using that $k\leq n$) we have that the second factor must be positive which implies that
	\[
	\Sigma_1 \geq \frac{|X_3|(n-|\cR_P|-2N)}{10} \geq \frac{|X_3|n(1-\alpha)}{10}(1-o(1)),
	\]
	as required.
\end{proofclaim}
	
	We are now ready to prove that $p_2=O(n^{-1})$.
	
	\begin{claim}\label{cl:p2UB}
		\[
		p_2 \leq \frac{20}{n(1-\alpha)}(1+o(1)).
		\]
	\end{claim}
	\begin{proofclaim}
		We prove the claim by using a combination of row cycle switching and column cycle switching, and we consider switching from squares in $X_3$ to squares in $Y_2 \setminus X_3$. 
		
		Let $L \in X_3$ and for $i \in \{0, 1, 2, 3\}$, let $r_i = r_i(L)$, $c_i = c_i(L)$, and $s_i = s_i(L)$. Let $k \in [|\cR_P|+N, n]$ be such that $L \in A_k$ and let $x \in [n] \setminus (\cR_P \cup G_L \cup \{r_0, r_1, r_2, r_3\})$. Note that there are $k-|\cR_P|-N$ choices for $x$. Let $\rho = \rho_L(r_3, x, c_3)$. If $\rho$ does not hit column $c_2$, then let $L'$ be the Latin square obtained from $L$ by switching on $\rho$. If $\rho$ does hit column $c_2$, then let $L''$ be the Latin square obtained from $L$ by switching on $\gamma_L(c_2, c_3, x)$ and let $L'$ be the Latin square obtained from $L''$ by switching on $\rho_{L''}(r_3, x, c_3)$. We claim that either way, $L' \in Y_2 \setminus X_3$. 
		
		First consider when $\rho$ does not hit column $c_2$. Since $x \not\in \cR_P \cup \{r_0, r_1, r_2, r_3\}$, it follows that $\rho$ does not hit a cell in $Q(L)$. Therefore, $L' \in Y_2$. Furthermore, $r_i(L') = r_i$, $c_i(L') = c_i$, and $s_i(L') = s_i$ for all $i \in \{0, 1, 2, 3\}$ except that $s_3(L') = L[x, c_3] \neq L[r_3, c_3] = s_2 = s_2(L')$. It follows that $L' \in Y_2 \setminus X_3$. 
		
		Now suppose that $\rho$ does hit column $c_2$. Since $x \notin \cR_P \cup G_L \cup \{r_0, r_1\}$, it follows that $\gamma_L(c_2, c_3, x) \neq \gamma_L$ and since $x \not\in \{r_2, r_3\}$, it follows that $\gamma_L(c_2, c_3, x) \neq \gamma_L(c_2, c_3, r_2)$ (recalling that $\gamma_L(c_2, c_3, r_2)$ is an intercalate). As we also have that $\{c_2, c_3\} \cap \cC_P = \emptyset$ due to how we chose $c_2$ and $c_3$ in our defining procedure at the start of the proof, it follows that $\gamma_L(c_2, c_3, x)$ does not hit any cell in $Q(L) \cup \{(r_3, c_3)\}$. Therefore $L'' \in Y_2$, $r_i(L'') = r_i$, $c_i(L'') = c_i$, and $s_i(L'') = s_i$ for all $i \in \{0, 1, 2, 3\}$. Thus, $L'' \in X_3$. Also, \cref{l:col_cyc_switch} \eqref{i:row both intersect} implies that $\rho_{L''}(r_3, x, c_3)$ does not hit column $c_2$. Hence we can now argue as in the previous case to conclude that $L' \in Y_2 \setminus X_3$.
		
		Since there were $k-|\cR_P|-N$ choices for $x$, it follows that the total number of switches from $X_3$ to $Y_2 \setminus X_3$ is at least 
		\begin{equation}\label{e:Y2minusX3LB}
		\sum_{k=|\cR_P|+N}^{n} a_k(k-|\cR_P|-N) = \Sigma_1 \geq \frac{|X_3|n(1-\alpha)}{10}(1-o(1)),
		\end{equation}
	by \cref{cl:Sigma_1}.
		
		Now let $L' \in Y_2 \setminus X_3$ and suppose that $L'$ is obtained from some $L \in X_3$ by:
		\begin{enumerate}[(i)]
			\item switching on $\rho_L(r_3(L), x, c_3(L))$ for some $x \in [n] \setminus (G_L \cup \cR_P \cup \{r_0(L), r_1(L), r_2(L), r_3(L)\})$ such that $\rho_L(r_3(L), x, c_3(L))$ does not hit column $c_2(L)$, or
			\item switching on $\gamma_L(c_2(L), c_3(L), x)$ for some $x \in [n] \setminus (G_L \cup \cR_P \cup \{r_0(L), r_1(L), r_2(L), r_3(L)\})$ such that $\rho_L(r_3(L), x, c_3(L))$ does hit column $c_2(L)$ and then switching the resulting Latin square $L''$ on $\rho_{L''}(r_3(L''), x, c_3(L''))$. In this case, as we have seen, $L'' \in X_3$ and $\rho_{L''}(r_3(L''), x, c_3(L''))$ does not hit column $c_2(L'')$. 
		\end{enumerate} 
		Let $r''$ be the row such that $L'[r'', c_3(L')] = s_2(L')$ and note that $r'' \neq r_3(L')$ since $L' \in Y_2 \setminus X_3$. 
		
		First suppose that (i) holds. By \cref{f:row_cyc_reverse}, $L'$ is the Latin square obtained from $L$ by switching on $\rho_{L'}(r_3(L), x, c_3(L))$. Since $x \notin \cR_P \cup \{r_0(L), r_1(L), r_2(L), r_3(L)\}$ and $\rho_L(r_3(L), x, c_3(L))$ does not hit column $c_2(L)$, it follows that $\rho_L$ does not hit any cell in $Q(L)$. Therefore, $r_i(L') = r_i(L)$, $c_i(L') = c_i(L)$, and $s_i(L') = s_i(L)$ for all $i \in \{0, 1, 2, 3\}$ except that $s_3(L') = L[x, c_3(L)] \neq L[r_3(L), c_3(L)] = s_3(L)$. Furthermore, $L'[x, c_3(L)] = L[r_3(L), c_3(L)] = s_2(L) = s_2(L')$. Hence $x = r''$ and thus $L$ is uniquely determined.
		
		Now suppose that (ii) holds. Using the same arguments as above, $L''$ is uniquely determined: it is the Latin square obtained from $L'$ by switching on $\rho_{L'}(r_3(L'), r'', c_3(L'))$. Furthermore  $r_i(L'') = r_i(L')$, $c_i(L'') = c_i(L')$, and $s_i(L'') = s_i(L')$ for all $i \in \{0, 1, 2, 3\}$ except that $s_3(L'') \neq s_3(L')$.
        Since $x=r''\not\in G_L \cup \cR_P \cup \{r_0(L), r_1(L), r_2(L), r_3(L)\}$ and $\{c_2(L), c_3(L)\} \cap \cC_P = \emptyset$, it follows that $\gamma_L(c_2(L), c_3(L), x)$ does not hit any cell in $Q(L) \cup \{(r_3(L), c_3(L))\}$ and so $r_i(L) = r_i(L'') = r_i(L')$, $c_i(L) = c_i(L'') = c_i(L')$, and $s_i(L) = s_i(L'') = s_i(L')$ for all $i \in \{0, 1, 2, 3\}$. Therefore, by \cref{f:col_cyc_reverse}, $L$ is uniquely determined: it is the square obtained from $L''$ by switching on $\gamma_{L''}(c_2(L'), c_3(L'), r'')$. 
		
		It follows that there are at most $2$ possibilities for the square $L \in X_3$. Thus, the number of switches from $X_3$ to $Y_2 \setminus X_3$ is at most
		\begin{equation}\label{e:Y2minusX3UB}
			2|Y_2 \setminus X_3|.
		\end{equation}
	
	By combining \cref{e:Y2minusX3LB} and \cref{e:Y2minusX3UB}, we obtain
	\[
	p_2 = \frac{|X_3|}{|Y_2|} = \frac{|X_3|}{|X_3|+|Y_2 \setminus X_3|} \leq \frac1{1+\frac{n(1-\alpha)(1-o(1))}{20}} = \frac{20}{n(1-\alpha)}(1+o(1)),
	\]
	as required.
	\end{proofclaim}

	By combining \cref{cl:p2UB} and \cref{cl:pjrecurrence}, we obtain
	\begin{equation}\label{e:D_potential}
	p_0 \leq \frac{22+2\alpha^2-4\alpha+\alpha\beta-\beta}{n(1-\alpha)(1-2\alpha-\beta)^2}(1+o(1)),
	\end{equation}
	which proves the theorem.
\end{proof}

\section{Proof of main result}\label{s:main}

In this section, we prove \cref{t:main}. The theorem will follow easily by combining \cref{l:UB} and \cref{l:LB} below. \cref{l:UB} allows us to deal with the upper bound in \cref{t:main} and \cref{l:LB} allows us to deal with the lower bound.

\begin{lem}\label{l:UB}
	Let $\bL$ be a random Latin square of order $n$ and let $P$ be a partial Latin square of order $n$. Let $\mathscr{R} \subseteq [n]$ be a set of rows, let $c \in [n]$ be a column, and let $s \in [n]$ be a symbol such that $(r, c, s) \notin P$ for all $r \in \mathscr{R}$. Let $\alpha, \alpha', \beta > 0$ be such that $|\cR_P|+1 \leq \alpha n$, $|\cR_{P} \cup \mathscr{R}| \leq \alpha' n$, and $|\cC_P| \leq \beta n$ and suppose that $\alpha+\alpha'+\beta<1$.
    Then
	\[
	\Pr(\exists r \in \mathscr{R} : \bL[r, c]=s| \bL \supseteq P) \leq \frac{|\mathscr{R}|(D+1)}{n(1-\alpha-\alpha'-\beta)+|\mathscr{R}|(D+1)}
	\]
    where $D = D(\alpha, \beta)$ from \cref{t:interc}. In particular, when $\mathscr{R}=\{r\}$,
    \[
    \Pr(\bL[r, c]=s | \bL \supseteq P) \leq \frac\Delta n
    \]
    for some constant $\Delta = \Delta(\alpha, \beta)$ that satisfies $\Delta \to 23$ as $\alpha, \beta \to 0$.
\end{lem}
\begin{proof}
	Let $X$ be the set of Latin squares of order $n$ containing $P$ and let $Y \subseteq X$ be the set of squares $L \in X$ containing $P$ such that $L[r, c]=s$ for some $r \in \mathscr{R}$. We prove the lemma using $\eta$-switching, where we consider switching from squares in $Y$ to squares in $X \setminus Y$.
	
	Let $L \in Y$ and let $r \in \mathscr{R}$ be such that $L[r, c]=s$. Let $r_1 \in [n] \setminus (\cR_P \cup \mathscr{R})$. Let $c_1$ be the column such that $L[r, c_1] = L[r_1, c]$ and let $r_2$ be the row such that $L[r_2, c_1] = s$. There are at most $|\cC_P|$ choices of $r_1$ such that $c_1 \in \cC_P$ and there are at most $|\cR_P|$ choices of $r_1$ such that $r_2 \in \cR_P$. Thus, there are at least $n-|\cR_P \cup \mathscr{R}|-|\cR_P|-|\cC_P|$ choices of $r_1 \notin \mathscr{R}$ such that $\eta_L(r, r_1, c) \cap P = \emptyset$. For any such $r_1$, switching $L$ on $\eta_L(r, r_1, c)$ yields a square $L' \in X$. Furthermore, $L' \in X \setminus Y$, since $L'[r_1, c] = L[r, c] = s$ and $r_1 \notin \mathscr{R}$. Hence, the total number of switches from $Y$ to $X \setminus Y$ is at least
	\begin{equation}\label{e:YLB}
		|Y|(n-|\cR_P \cup \mathscr{R}|-|\cR_P|-|\cC_P|) \geq |Y|n(1-\alpha-\alpha'-\beta).
	\end{equation}
	
	For $L' \in X \setminus Y$, let $r'(L')$ be the row such that $L'[r'(L'), c] = s$. For $r \in \mathscr{R}$, define $X'[r] \subseteq X \setminus Y$ to be the set of squares $L' \in X \setminus Y$ such that $\eta_{L'}(r, r'(L'), c)$ is of Type One and define $X''[r] = (X \setminus Y) \setminus X'[r]$. Let $r \in \mathscr{R}$ and let $L' \in X \setminus Y$. If $L \in Y$ is such that $L'$ is obtained from $L$ by switching on $\eta_L(r, r_1, c)$ for some row $r_1$, then $r_1 = r'(L')$.
	By \cref{l:eta_reverse}, if $L' \in X'[r]$, then there are at most $n$ squares $L \in Y$ from which $L'$ can be obtained by switching on $\eta_L(r, r'(L'), c)$ and if $L' \in X''[r]$, then there is at most one possible square $L \in Y$. Thus, the total number of switches from $Y$ to $X \setminus Y$ is at most
	\begin{equation}\label{e:XminusYUBinit}
    \sum_{r \in \mathscr{R}} (n|X'[r]|+|X''[r]|) \leq |X \setminus Y|\sum_{r \in \mathscr{R}} \left(n\frac{|X'[r]|}{|X \setminus Y|}+1\right).
	\end{equation}
    We next upper bound $|X'[r]|/|X \setminus Y|$ for all $r \in \mathscr{R}$. Let $r \in \mathscr{R}$. For $r' \in [n] \setminus \mathscr{R}$ and $s' \in [n] \setminus \{s\}$, let $Z(r', s')$ be the set of squares $L' \in X \setminus Y$ such that $r'(L') = r'$ and $L'[r, c]=s'$. Note that $X \setminus Y$ is the disjoint union
    \[
    \bigcup_{\substack{r' \in [n] \setminus \mathscr{R}\\s' \in [n] \setminus \{s\}}} Z(r', s')
    \]
    and so
    \begin{equation}\label{e:x'_bound}
    \frac{|X'[r]|}{|X \setminus Y|} = \Pr(\bL \in X'[r] | \bL \in X \setminus Y) = \sum_{\substack{r' \in [n] \setminus \mathscr{R}\\s' \in [n] \setminus \{s\}}} \Pr(\bL \in X'[r] | \bL \in Z(r', s'))\Pr(\bL \in Z(r', s')| \bL \in X \setminus Y).
    \end{equation}
    For $r' \in [n] \setminus \mathscr{R}$ and $s' \in [n] \setminus \{s\}$, define
    \[
    P(r', s') = P \cup \{(r', c, s), (r, c, s')\}.    
    \]
    Then $Z(r', s')$ is exactly the set of squares in $X \setminus Y$ that contain $P(r', s')$. Therefore by \cref{t:interc},
    \begin{equation}\label{e:x'_bound_2}
        \Pr(\bL \in X'[r] | \bL \in Z(r', s')) = \Pr(\eta_{\bL}(r, r', c) \text{ is of Type One} | \bL \supseteq P(r', s')) \leq \frac Dn,
    \end{equation}
    where $D = D(\alpha, \beta)$ from \cref{t:interc}. Combining \cref{e:x'_bound} and \cref{e:x'_bound_2}, we see that
	\[
	\frac{|X'[r]|}{|X \setminus Y|} \leq \frac{D}n
	\]
    for all $r \in \mathscr{R}$.
	Then by \cref{e:XminusYUBinit}, the total number of switches from $Y$ to $X \setminus Y$ is at most
	\begin{equation}\label{e:XminusYUB}
		|\mathscr{R}||X \setminus Y|(D+1).
	\end{equation}
	
	By combining \cref{e:YLB} and \cref{e:XminusYUB}, we obtain
	\begin{align}
		\Pr(\exists r \in \mathscr{R} : \bL[r, c]=s | \bL \supseteq P) &= \frac{|Y|}{|Y|+|X \setminus Y|} \nonumber\\
		&\leq \frac1{1+\frac{n(1-\alpha-\alpha'-\beta)}{|\mathscr{R}|(D+1)}} \nonumber\\
		&= \frac{|\mathscr{R}|(D+1)}{n(1-\alpha-\alpha'-\beta)+|\mathscr{R}|(D+1)}\label{e:Delta_potential},
	\end{align}
	which proves the first claim of the lemma. For the second claim, notice that if $\mathscr{R}=\{r\}$, then by \cref{e:Delta_potential},
    \[
    \Pr(\bL[r, c]=s | \bL \supseteq P) \leq \frac{(D+1)}{n(1-2\alpha-\beta)+D+1} = \frac{(D+1)}{n(1-2\alpha-\beta)}(1+o(1)).
    \]
    The claim then follows, since $D \to 22$ as $\alpha, \beta \to 0$, by \cref{t:interc}.
\end{proof}

Using analogous arguments, we can prove the following lemma.

\begin{lem}\label{l:UB2}
	Let $\bL$ be a random Latin square of order $n$ and let $P$ be a partial Latin square of order $n$. Let $\mathscr{C} \subseteq [n]$ be a set of columns, let $r \in [n]$ be a row, and let $s \in [n]$ be a symbol such that $(r, c, s) \notin P$ for all $c \in \mathscr{C}$. Let $\alpha, \beta' > 0$ be such that $|\cR_P| \leq \alpha n$ and $|\cC_P \cup \mathscr{C}| \leq \beta'n$ and suppose that $2\alpha+\beta'<1$.
    Then
	\[
	\Pr(\exists c \in \mathscr{C} : \bL[r, c]=s| \bL \supseteq P) \leq \frac{|\mathscr{C}|(D+1)}{n(1-2\alpha-\beta')+|\mathscr{C}|(D+1)}
	\]
    where $D = D(\alpha, \beta')$ from \cref{t:interc}. 
\end{lem}

We are now ready to prove our lower bound probability. 

\begin{lem}\label{l:LB}
	Let $\bL$ be a random Latin square of order $n$ and let $P$ and $P'$ be distinct partial Latin squares of order $n$ such that $P' = P \cup \{(r, c, s)\}$. Let $\alpha, \beta > 0$ be such that $|\cR_{P'}| \leq \alpha n$ and $|\cC_{P'}| \leq \beta n$ and suppose that $2\alpha+\beta<1$. Then
	\[
	\Pr(\bL \supseteq P' | \bL \supseteq P) \geq \frac{\delta}n
	\]
    for some positive constant $\delta = \delta(\alpha, \beta)$ that satisfies $\delta \to 1/23$ as $\alpha, \beta \to 0$.
\end{lem}
\begin{proof}
	Let $X$ be the set of Latin squares of order $n$ containing $P$ and let $Y \subseteq X$ be the set of squares in $X$ containing $P'$. We prove the lemma using $\eta$-switching, where we consider switching from squares in $X \setminus Y$ to squares in $Y$.
	
	For $L \in X \setminus Y$, let $s'(L) = L[r, c]$, let $r'(L)$ be the row such that $L[r'(L), c] = s$, let $c'(L)$ be the column such that $L[r, c'(L)] = s$, and let $r''(L)$ be the row such that $L[r''(L), c'(L)] = s'(L)$. Define
    \begin{align*}
    &\mathcal{G}_1 = \{L \in X \setminus Y : r'(L) \notin \cR_{P'}\},\\
    &\mathcal{G}_2 = \{L \in X \setminus Y : c'(L) \notin \cC_{P'}\}, \text{ and}\\
    &\mathcal{G}_3 = \{L \in X \setminus Y : r''(L) \notin \cR_{P'}\}.
    \end{align*}
    Also define $\mathcal{G} = \mathcal{G}_1 \cap \mathcal{G}_2 \cap \mathcal{G}_3$. The first step in our proof is to lower bound $p = \Pr(\bL \in \mathcal{G} | \bL \in X \setminus Y)$. By the chain rule of probability, 
    \begin{equation}\label{e:p}
    p = \Pr(\bL \in \cG_1 | \bL \in X \setminus Y)\Pr(\bL \in \cG_2 | \bL \in \cG_1)\Pr(\bL \in \cG_3 | \bL \in \cG_1 \cap \cG_2).
    \end{equation}
    We will bound the three probabilities in the right hand side of \cref{e:p} separately.

      We now lower bound \[p_1 = \Pr(\bL \in \cG_1 | \bL \in X \setminus Y)=\frac{|\cG_1\cap X\setminus Y|}{|X\setminus Y|}=\frac{|\cG_1\cap X|}{|X\setminus Y|}\geq \frac{|\cG_1\cap X|}{|X|}= \Pr(\bL \in \cG_1 | \bL \in X ) ,\]
    where we used that $\cG_1\cap Y=\emptyset$. 
   Therefore
    \begin{equation}\label{e:1-p1_2}
    1-p_1 \leq \Pr(\exists r_1 \in \cR_{P'} : \bL[r_1, c]=s | \bL \supseteq P).
    \end{equation}
    Let $D = D(\alpha, \beta)$ from \cref{t:interc}. Apply \cref{l:UB} with $\mathscr{R} = \cR_{P'}$ and combine the conclusion with \cref{e:1-p1_2} to obtain, 
    \[
    1-p_1 \leq \frac{|\cR_{P'}|(D+1)}{n(1-2\alpha-\beta)+|\cR_{P'}|(D+1)} \leq \frac{\alpha(D+1)}{1-2\alpha-\beta+\alpha(D+1)},
    \]
    meaning that
    \begin{equation}\label{e:p1_LB}
    p_1 \geq 1-\frac{\alpha(D+1)}{1-2\alpha-\beta+\alpha(D+1)} = \frac{1-2\alpha-\beta}{1-2\alpha-\beta+\alpha(D+1)}.
    \end{equation}

    Next, we lower bound $p_2 = \Pr(\bL \in \cG_2 | \bL \in \cG_1)$. For $r' \in [n] \setminus \cR_{P'}$, let 
    \[
    \cG(r') = \{L \in \cG_1 : r'(L) = r'\}
    \]
    so that we can partition
    \[
    \cG_1 = \bigcup_{r' \in [n] \setminus \cR_{P'}} \cG(r').
    \]
    Then
    \begin{equation}\label{e:p2_sum}
    p_2 = \sum_{r' \in [n] \setminus \cR_{P'}} \Pr(\bL \in \cG_2 | \bL \in \cG(r'))\Pr(\bL \in \cG(r')| \bL \in \cG_1).
    \end{equation}
    We will lower bound $p_2(r') = \Pr(\bL \in \cG_2 | \bL \in \cG(r'))$ for each $r' \in [n] \setminus \cR_{P'}$. Let $r' \in [n] \setminus \cR_{P'}$. Set $P(r') = P \cup \{(r', c, s)\}$ and set $\mathscr{C} = \cC_{P'}$ so that 
    \begin{equation}\label{e:1-p2}
    1-p_2(r') = \Pr(\bL \notin \cG_2 | \bL \in \cG(r')) = \Pr(\exists c_1 \in \mathscr{C} : \bL[r, c_1]=s | \bL \supseteq P(r')).
    \end{equation}
    Let $\alpha_1 = \alpha+1/n$ so that $|\cR_{P(r')}| \leq \alpha_1n$.
    By \cref{l:UB2} and \cref{e:1-p2},
    \begin{equation}\label{e:1-p2_2}
    1-p_2(r') \leq \frac{|\mathscr{C}|(D'+1)}{n(1-2\alpha_1-\beta)+|\mathscr{C}|(D'+1)} \leq \frac{\beta(D'+1)}{1-2\alpha_1-\beta+\beta(D'+1)}
    \end{equation}
    where $D' = D(\alpha_1, \beta)$. From \cref{e:D_potential}, we may assume that
    \begin{equation}\label{e:D_cont}
    D(\alpha, \beta) = \frac{22+2\alpha^2-4\alpha+\alpha\beta-\beta}{(1-\alpha)(1-2\alpha-\beta)^2}(1+o(1)),
    \end{equation}
    which means that
    \[
    D' = D(1+o(1)).
    \]
    Combining this with \cref{e:1-p2_2}, we obtain
    \[
    p_2(r') \geq 1-\frac{\beta(D+1)}{1-2\alpha-\beta+\beta(D+1)}(1+o(1)) = \frac{1-2\alpha-\beta}{1-2\alpha-\beta+\beta(D+1)}(1-o(1))
    \]
    for each $r' \in [n] \setminus \cR_{P'}$. It follows from \cref{e:p2_sum} that
    \begin{equation}\label{e:p2_LB}
    p_2 \geq \frac{1-2\alpha-\beta}{1-2\alpha-\beta+\beta(D+1)}(1+o(1)).
    \end{equation}

    Finally, we lower bound $p_3 = \Pr(\bL \in \cG_3 | \bL \in \cG_1 \cap \cG_2)$. Let $T = [n] \times ([n] \setminus \cR_{P'}) \times ([n] \setminus \cC_{P'})$ and for $(s', r', c') \in T$, let 
    \[
    \cG(s', r', c') = \{L \in X \setminus Y : s'(L) = s', r'(L) = r', \text{ and } c'(L) = c'\}
    \]
    so that we can partition 
    \[
    \cG_1 \cap \cG_2 = \bigcup_{(s', r', c') \in T} \cG(s', r', c').
    \]
    Then
    \begin{equation}\label{e:p3_sum}
    p_3 = \sum_{(s', r', c') \in T} \Pr(\bL \in \cG_3 | \bL \in \cG(s', r', c'))\Pr(\bL \in \cG(s', r', c')| \bL \in \cG_1 \cap \cG_2).
    \end{equation}
    We will lower bound $p_3(s', r', c') = \Pr(\bL \in \cG_3 | \bL \in \cG(s', r', c'))$ for each $(s', r', c') \in T$. Let $(s', r', c') \in T$. Set $P(s', r', c') = P \cup \{(r, c, s'), (r', c, s), (r, c', s)\}$ and set $\mathscr{R} = \cR_{P'}$ so that
    \begin{equation}\label{e:p_3}
    1-p_3(s', r', c') = \Pr(\bL \notin \cG_3 | \bL \in \cG(s', r', c')) = \Pr(\exists r_2 \in \mathscr{R} : \bL[r_2, c'] = s' | \bL \supseteq P(s', r', c')).
    \end{equation}
    Let $\beta_1=\beta+1/n$ so that $|\cR_{P(s', r', c')}| \leq \alpha_1n$ and $|\cC_{P(s', r', c')}| \leq \beta_1n$. By \cref{l:UB} and \cref{e:p_3},
    \[
    1-p_3(s', r', c') \leq \frac{|\mathscr{R}|(D(\alpha_1, \beta_1)+1)}{n(1-2\alpha_1-\beta_1)+|\mathscr{R}|(D(\alpha_1, \beta_1)+1)} \leq \frac{\alpha(D+1)}{1-2\alpha-\beta+\alpha(D+1)}(1+o(1)),
    \]
    where we have again used \cref{e:D_cont}. Hence,
    \[
    p_3(s', r', c') \geq 1-\frac{\alpha(D+1)}{1-2\alpha-\beta+\alpha(D+1)}(1+o(1)) = \frac{1-2\alpha-\beta}{1-2\alpha-\beta+\alpha(D+1)}(1-o(1))
    \]
    for every $(s', r', c') \in T$. Thus, by \cref{e:p3_sum},
    \begin{equation}\label{e:p3_LB}
    p_3 \geq \frac{1-2\alpha-\beta}{1-2\alpha-\beta+\alpha(D+1)}(1-o(1)).
    \end{equation}
    
    Finally, by combining \cref{e:p}, \cref{e:p1_LB}, \cref{e:p2_LB}, and \cref{e:p3_LB}, we obtain
    \begin{equation}\label{e:p_final_bound}
    p \geq \frac{(1-2\alpha-\beta)^3}{(1-2\alpha-\beta+\alpha(D+1))^2(1-2\alpha-\beta+\beta(D+1))}(1-o(1)).
    \end{equation}
    Denote the right hand side of \cref{e:p_final_bound} by $\Delta'$ and note that $\Delta' \to 1$ as $\alpha, \beta \to 0$.

    We now count the number of switches from $X \setminus Y$ to $Y$. 
	Let $L \in \cG$ and let $r' = r'(L)$.
    The definition of $\cG$ ensures that $\eta_L(r, r', c) \cap P = \emptyset$. Therefore, the Latin square $L'$ obtained from $L$ by switching on $\eta_L(r, r', c)$ lies in $X$. Furthermore, $L'[r, c] = L[r', c] = s$, thus $L' \in Y$. Hence, the total number of switches from $X \setminus Y$ to $Y$ is at least
	\begin{equation}\label{e:XminusYLB}
		|\cG| \geq |X \setminus Y|p \geq |X \setminus Y|\Delta'.
	\end{equation}
	
	For $r' \in [n]$, let $Y[r']$ be the set of Latin squares $L
	'\in Y$ such that $\eta_{L'}(r, r', c)$ is of Type One and let $Y'[r'] = Y \setminus Y[r']$. Let $r' \in [n]$. By \cref{l:eta_reverse}, if $L' \in Y[r']$, then there are at most $n$ squares $L \in X \setminus Y$ from which $L'$ has been obtained by switching on $\eta_L(r, r', c)$ and if $L' \in Y'[r']$, then there is at most one such square $L \in X \setminus Y$. Thus, the total number of switches from $X \setminus Y$ to $Y$ is at most
	\begin{equation}\label{e:YUBinit}
		\sum_{r' \in [n]} (n|Y[r']|+|Y'[r']|) \leq |Y|\sum_{r' \in [n]} \left(n\frac{|Y[r']|}{|Y|}+1\right).
	\end{equation}
	By \cref{t:interc},
	\[
	\frac{|Y[r']|}{|Y|} \leq \frac{D}n
	\]
	for every $r' \in [n]$. Thus, by \cref{e:YUBinit}, the total number of switches from $X \setminus Y$ to $Y$ is at most
	\begin{equation}\label{e:YUB}
		|Y|n(D+1).
	\end{equation}
	
	By combining \cref{e:XminusYLB} and \cref{e:YUB}, we obtain
	\begin{equation}\label{e:p_LB_init}
		\Pr(\bL \supseteq P' | \bL \supseteq P) = \frac{|Y|}{|Y|+|X \setminus Y|}
		\geq \frac1{1+\frac{n(D+1)}{\Delta'}} 
		= \frac{\Delta'}{n(D+1)}(1-o(1)),
	\end{equation}
    which proves the lemma, since $D \to 22$ as $\alpha, \beta \to 0$ by \cref{t:interc}.
\end{proof}

We are now ready to prove \cref{t:main}.

\begin{proof}[Proof of \cref{t:main}]
	Write $P = \{q_1, q_2, \ldots, q_N\}$. Let $P_0 = \emptyset$ and for $i \in [N]$, let $P_i = \{q_j : j \in [i]\}$. By the chain rule of probability,
	\begin{equation}\label{e:CR}
		\Pr(\bL \supseteq P) = \prod_{i=1}^{N} \Pr(\bL \supseteq P_i | \bL \supseteq P_{i-1}).
	\end{equation}
	The upper bound now follows from \cref{e:CR} and \cref{l:UB} and the lower bound follows from \cref{e:CR} and \cref{l:LB}.
\end{proof}

\section{Applications} \label{s:applications}

In this section, we demonstrate the usefulness of \cref{t:main} by applying it to prove our results on subsquares, subrectangles, and isotopic copies in Latin squares. We need the following lemma estimating the conditional probability that a random Latin square contains a partial Latin square $P'$, given that it contains some partial Latin square $P$, when $P'$ is obtained from $P$ by adding an entry in an otherwise unused row. In the case where $|\cR_P|, |\cC_P|=o(n)$, it gives the asymptotically correct probability. 

\begin{lem}\label{l:newrcs}
	Let $\bL$ be a random Latin square of order $n$ and let $P$ and $P'$ be distinct partial Latin squares of order $n$ such that $P' = P \cup \{(r, c, s)\}$ and $r \notin \cR_P$. Let $\alpha, \beta>0$ be such that $|\cR_{P'}| \leq \alpha n$ and $|\cC_{P'}| \leq \beta n$.
    Then
    \[
    \frac{1-2\alpha-\beta}{n(1-2\alpha-\beta+\alpha(D+1))}(1-o(1)) \leq \Pr(\bL \supseteq P' | \bL \supseteq P) \leq \frac1{n(1-\alpha)}(1+o(1)),
    \]
      where $D = D(\alpha, \beta)$ from \cref{t:interc}. 
\end{lem}
\begin{proof}
    Let $X$ be the set of Latin squares of order $n$ that contain $P$ and let $Y \subseteq X$ be the set of squares containing $P'$. We prove the lemma using row cycle switching, where we consider switching from squares $L \in Y$ on substructures $\rho_L(r, r', c)$ with $r' \notin \cR_{P'}$ to squares in $X \setminus Y$.

    Let $L \in Y$ and let $r' \in [n] \setminus \cR_{P'}$. Let $L'$ be the Latin square obtained from $L$ by switching on $\rho_L(r, r', c)$. Since $\{r, r'\} \cap \cR_P = \emptyset$, it follows that $L' \in X$. Furthermore, $L'[r, c] = L[r', c] \neq L[r, c] = s$, meaning that $L' \in X \setminus Y$. Hence, there are at least $n-|\cR_{P'}| \geq n(1-\alpha)$ switches from each $L \in Y$ to $X \setminus Y$, meaning that the total number of switches is at least
    \begin{equation}\label{e:num_switch_Y_LB}
        |Y|n(1-\alpha).
    \end{equation}
    On the other hand, there are trivially at most $n$ choices for $r'$ for each $L \in Y$, meaning that the number of switches is at most
    \begin{equation}\label{e:num_switch_Y_UB}
        |Y|n.
    \end{equation}

    For $L' \in X \setminus Y$, let $r'(L')$ be the row such that $L'[r'(L'), c]=s$. Let $X' \subseteq X$ be the set of squares $L'$ such that $r'(L') \in [n] \setminus \cR_{P'}$. By \cref{l:UB} with $\mathscr{R} = \cR_{P'}$,
    \begin{align}
    \frac{|X'|}{|X \setminus Y|} &\geq \frac{|X'|}{|X|} \nonumber\\
    &= 1-\Pr(\exists r' \in \cR_{P'} : \bL[r', c] = s | \bL \supseteq P) \nonumber\\
    &\geq 1-\frac{\alpha n(D+1)}{n(1-2\alpha-\beta)+\alpha n(D+1)} \nonumber\\
    &= \frac{1-2\alpha-\beta}{1-2\alpha-\beta+\alpha(D+1)},\label{e:X'_LB}
    \end{align}
    where $D = D(\alpha, \beta)$ from \cref{t:interc}. 

    Let $L' \in X \setminus Y$ and let $r' = r'(L')$. Suppose that $L \in Y$ is such that $L'$ is obtained from $L$ by switching on $\rho_L(r, r', c)$ for some $r' \in [n] \setminus \cR_{P'}$. Then $r' = r'(L')$. Thus, there is at most one choice of $L$ for each $L' \in X \setminus Y$, meaning that the total number of switches from $Y$ to $X \setminus Y$ is at most
    \begin{equation}\label{e:num_switch_XminusY_UB}
        |X \setminus Y|.
    \end{equation}
    Now suppose that $L' \in X'$ and let $L$ be the Latin square obtained from $L'$ by switching on $\rho_{L'}(r, r'(L'), c)$. Then by \cref{f:row_cyc_reverse}, $L'$ is the Latin square obtained from $L$ by switching on $\rho_L(r, r'(L'), c)$ and $r'(L') \notin \cR_{P'}$. It follows that there is exactly one choice of $L \in Y$ from which $L'$ is obtained by switching on $\rho_L(r, r', c)$ for some $r' \in [n] \setminus \cR_{P'}$ when $L' \in X'$. Hence, the total number of switches from $Y$ to $X \setminus Y$ is at least
    \begin{equation}\label{e:num_switch_XminusY_LB}
        |X'| \geq |X \setminus Y|\frac{1-2\alpha-\beta}{1-2\alpha-\beta+\alpha(D+1)},
    \end{equation}
    by \cref{e:X'_LB}.

    By combining \cref{e:num_switch_Y_LB} and \cref{e:num_switch_XminusY_UB}, we obtain
    \[
    \Pr(\bL \supseteq P' | \bL \supseteq P) = \frac{|Y|}{|Y|+|X \setminus Y|} \leq \frac1{1+n(1-\alpha)} = \frac1{n(1-\alpha)}(1+o(1)).
    \]
    Similarly, by combining \cref{e:num_switch_Y_UB} and \cref{e:num_switch_XminusY_LB}, we obtain
    \[
    \Pr(\bL \supseteq P' | \bL \supseteq P) = \frac{|Y|}{|Y|+|X \setminus Y|} \geq \frac1{1+\frac{n(1-2\alpha-\beta+\alpha(D+1))}{1-2\alpha-\beta}} = \frac{1-2\alpha-\beta}{n(1-2\alpha-\beta+\alpha(D+1))}(1-o(1)),
    \]
    completing the proof.
\end{proof}

\begin{rem}\label{rem:symm}
    Using the symmetry between the rows, columns, and symbols, a statement similar to \cref{l:newrcs} is true if instead of assuming that $r \notin \cR_P$, we assume that $c \notin \cC_P$ or $s \notin \cS_P$. In particular, if $|\cR_P|, |\cC_P|, |\cS_P| = o(n)$ and $r \notin \cR_P$, $c \notin \cC_P$, or $s \notin \cS_P$, then $\Pr(\bL \supseteq P' | \bL \supseteq P) = (1+o(1))/n$.
    This can be proved using analogous arguments as in the above proof but with column cycle switching or symbol cycle switching in place of row cycle switching, accordingly.
\end{rem}

\subsection{Subsquares and subrectangles}\label{s:subsq}

This subsection is dedicated to the proofs of \cref{t:E3n}, \cref{t:Emn_asymptotics}, and \cref{t:Eabn}. We will need the following notation. For a partial Latin square $P$ of order $n$ and subsets $\cR, \cC \subseteq [n]$, we denote by $P[\cR, \cC]$ the submatrix of $P$ induced by the rows in $\cR$ and the columns in $\cC$. We are now ready to prove \cref{t:E3n}.

\begin{proof}[Proof of \cref{t:E3n}]
	Let $T$ denote the set of all quadruples $(\cR, \cC, \cS, L)$ where $\cR$, $\cC$, and $\cS$ are subsets of $[n]$ of cardinality $3$ and $L$ is a Latin square of order $3$ with row index set $\cR$, column index set $\cC$, and symbol set $\cS$. For $(\cR, \cC, \cS, L) \in T$, let $P = P(\cR, \cC, \cS, L)$ denote the partial Latin square of order $n$ whose non-empty cells are precisely $\cR \times \cC$, that satisfies $P[\cR, \cC] = L$. Let $\bL$ be a random Latin square of order $n$. Then
	\[
	\mathbb{E}_3(n) = \sum_{(\cR, \cC, \cS, L) \in T} \Pr(\bL \supseteq P(\cR, \cC, \cS, L)).
	\]
	We want to lower bound $\Pr(\bL \supseteq P(\cR, \cC, \cS, L))$ for each $(\cR, \cC, \cS, L) \in T$.
	By symmetry, 
	\begin{equation}\label{e:E3symm}
		\Pr(\bL \supseteq P(\cR, \cC, \cS, L)) = \Pr(\bL \supseteq P(\cR', \cC', \cS', L'))
	\end{equation}
	for any $\{(\cR, \cC, \cS, L), (\cR', \cC', \cS', L')\} \subseteq T$ and so it suffices to lower bound $\Pr(\bL \supseteq P)$ where $P$ is the partial Latin square of order $n$ whose non-empty cells are precisely $[3]^2$ where $P[[3], [3]]$ is equal to the Latin square
	\[
	\begin{pmatrix}
		1&2&3\\
		2&3&1\\
		3&1&2
	\end{pmatrix}.
	\]
    Let $P_0 = \emptyset$ and for $i \in [9]$, let $P_i$ be obtained from $P_{i-1}$ by adding the $i$-th entry of the tuple
    \[
    ((1, 1, 1), (1, 2, 2), (2, 1, 2), (2, 2, 3), (1, 3, 3), (2, 3, 1), (3, 1, 3), (3, 2, 1), (3, 3, 2)).
    \]
	By the chain rule of probability,
	\begin{equation}\label{e:E3CR}
		\Pr(\bL \supseteq P) = \prod_{i=1}^{9} \Pr(\bL \supseteq P_i | \bL \supseteq P_{i-1}).
	\end{equation}
	For $i \in \{1, 2, 3, 4, 5, 7\}$,
	\begin{equation}\label{e:E3LB1}
		\Pr(\bL \supseteq P_i | \bL \supseteq P_{i-1}) \geq \frac{1-o(1)}n
	\end{equation}
	by \cref{l:newrcs} (or Remark~\ref{rem:symm}). For $i \in \{6, 8, 9\}$,
	\begin{equation}\label{e:E3LB2}
		\Pr(\bL \supseteq P_i | \bL \supseteq P_{i-1}) \geq \frac{1-o(1)}{23n}
	\end{equation}
	by \cref{l:LB}. Combining \cref{e:E3symm}, \cref{e:E3CR}, \cref{e:E3LB1}, and \cref{e:E3LB2}, we obtain
	\[
	\Pr(\bL \supseteq P(\cR, \cC, \cS, L)) \geq \left(\frac1n\right)^9\frac{1-o(1)}{23^3}
	\]
	for any $(\cR, \cC, \cS, L) \in T$. Since there are $\binom{n}{3}$ choices for each of $\cR$, $\cC$, and $\cS$ and there are $12$ choices for $L$ given any choice of $\cR$, $\cC$, and $\cS$, it follows that 
	\[
	\mathbb{E}_3(n) \geq 12\binom{n}{3}^3\left(\frac1{n}\right)^9\frac{1-o(1)}{23^3} = \frac1{219006}-o(1),
	\]
	as required.
\end{proof}

We now work towards proving \cref{t:Emn_asymptotics} and \cref{t:Eabn}.
In these proofs, we will distinguish two cases, depending on whether $a=o(n)$ or $a=\Omega(n)$. For the former range, the proofs will follow easily from \cref{t:main} but for the latter range, we employ a different approach. To estimate the probability of a random Latin square containing a specific subsquare (or subrectangle) with $a$ rows, we first estimate the probability that a random $a \times n$ Latin rectangle contains that specific subsquare (or subrectangle). We will then use the following theorem of McKay and Wanless~\cite{manysubsq}, which allows us to translate probabilities of events occurring in random Latin rectangles to events occurring in random Latin squares and is proven using permanent estimates, to translate this into a bound on the probability that we are interested in. 

\begin{thm}\label{t:ext}
    Let $k \leq n$ and let $L$ and $L'$ be $k \times n$ Latin rectangles. The ratio of the number of completions of $L$ to the number of completions of $L'$ is at most $\exp(O(n\log^2n))$.
\end{thm}

Let $k \leq n$ be a positive integer. For the remainder of this section, we will let $\mathscr{L}(k, n)$ denote the number of $k \times n$ Latin rectangles. Using permanent estimates in a standard way (as in, say, the proof of~\cite[Theorem 17.3]{num_squares}), we can easily obtain the bounds
\begin{equation}\label{e:Lab}
\left(\frac n{e^2}\right)^{kn} \leq \mathscr{L}(k, n) \leq n^{kn}.
\end{equation}
Being a bit more careful with estimates, one can obtain the following theorem of Luria and Simkin~\cite{num_rect}.

\begin{thm}\label{t:num_rect}
     Let $k = k(n)$ be a positive integer satisfying $\Omega(n) = k < n$ and let $\alpha = \alpha(n) = k/n$. 
     \[
     \mathscr{L}(k, n) = \left(\left(\frac1{1-\alpha}\right)^{(1-\alpha)/\alpha}\frac n{e^2}(1+o(1))\right)^{\alpha n^2}.
     \]
 \end{thm}

 Since \cref{t:num_rect} only gives an estimate on $\mathscr{L}(k, n)$ when $k<n$, we state the following result for the number $\mathscr{L}(n)$ of Latin squares of order $n$~\cite[Theorem $17.3$]{num_squares}.

\begin{thm}\label{t:num_squares}
    \[
    \mathscr{L}(n) = \left(\frac n{e^2}(1+o(1))\right)^{n^2}.
    \]
\end{thm}

Using the approach described after the proof of \cref{t:E3n}, we next estimate the probability of a random Latin square containing a specific $a \times b$ subrectangle when $a = \Omega(n)$. We are able to get the asymptotically correct value in the base of the exponent.

\begin{prop}\label{prop:rectangles}
    Suppose that $a = a(n) = \Omega(n)$ and $b = b(n) \geq a$ are such that $a+b < n$. Let $\alpha = \alpha(n) = a/n$ and let $\beta = \beta(n) = b/n$. Let $P$ be a partial Latin square of order $n$ whose non-empty cells form an $a \times b$ Latin rectangle and let $\bL$ be a random Latin square of order $n$. Then,
    \[
    \Pr(\bL \supseteq P) = \left(e^2\left(\frac{(1-\alpha)^{1-\alpha}(1-\beta)^{(1-\beta)^2}}{(1-\alpha-\beta)^{(1-\alpha-\beta)(1-\beta)}}\right)^{1/\alpha\beta}(1+o(1))\frac1n\right)^{|P|}
    \]
\end{prop}
\begin{proof}
    Let $L = P[\cR_P, \cC_P]$, which by assumption is an $a \times b$ Latin rectangle. Let $\cR = [a]$, let $\cC = \cS = [b]$, and let $L_1$ be an $a \times b$ Latin rectangle (with row indices $[a]$ and column indices and symbol set $[b]$). Let $P_1$ be the partial Latin square of order $n$ whose non-empty cells are precisely $\cR \times \cC$ that satisfies $P_1[\cR, \cC] = L_1$. Then by symmetry, 
    \begin{equation}\label{e:symm_prob}
    \Pr(\bL \supseteq P) = \Pr(\bL \supseteq P_1).
    \end{equation}
    We thus focus on determining $\Pr(\bL \supseteq P_1)$.

    Set $k=a$ and let $\bL_2$ be a random $k \times n$ Latin rectangle. Also let $P_2$ be the $k \times n$ partial Latin rectangle whose entries are precisely the entries of $P_1$. There is a bijection between the set of $k \times n$ Latin rectangles containing $P_2$ and the set of $k \times (n-b)$ Latin rectangles on symbol set $[n] \setminus [b]$ (which is well-defined, since $k \leq n-b$ by assumption), obtained by mapping a rectangle $L$ containing $P_2$ to its submatrix $L[[k], [n] \setminus [b]]$ (i.e., the submatrix obtained by removing the first $b$ columns). 
    Therefore by \cref{t:num_rect}, the number of $k \times n$ Latin rectangles containing $P_2$ is
    \[
    \left(\left(\frac {1-\beta}{1-\alpha-\beta}\right)^{(1-\alpha-\beta)/\alpha}\frac{(1-\beta)n}{e^2}(1+o(1))\right)^{\alpha(1-\beta)n^2}.
    \]
    On the other hand, \cref{t:num_rect} tells us that the number of $k \times n$ Latin rectangles is
    \[
   \left(\left(\frac 1{1-\alpha}\right)^{(1-\alpha)/\alpha}\frac n{e^2}(1+o(1))\right)^{\alpha n^2}.
    \]
    Therefore, 
    \begin{align}
        \Pr(\bL_2 \supseteq P_2) &= \frac{\left(\left(\frac {1-\beta}{1-\alpha-\beta}\right)^{(1-\alpha-\beta)/\alpha}\frac{(1-\beta)n}{e^2}(1+o(1))\right)^{\alpha(1-\beta)n^2}}{\left(\left(\frac 1{1-\alpha}\right)^{(1-\alpha)/\alpha}\frac n{e^2}(1+o(1))\right)^{\alpha n^2}} \nonumber\\
        &= \left(\left(\frac{(1-\alpha)^{1-\alpha}(1-\beta)^{(1-\beta)^2}}{(1-\alpha-\beta)^{(1-\alpha-\beta)(1-\beta)}}\right)^{1/\alpha\beta}\frac{e^2}n(1+o(1))\right)^{\alpha\beta n^2}.\label{e:L2P2}
    \end{align}
    
    Let $\bL_k$ be the submatrix $\bL[[k], [n]]$ of $\bL$ (i.e. the submatrix formed by the first $k$ rows of $\bL$) so that
    \begin{equation}\label{e:first_k_rows}
        \Pr(\bL \supseteq P_1) = \Pr(\bL_k \supseteq P_1).
    \end{equation}
    By \cref{t:ext} and \cref{e:L2P2},
    \begin{align}
    \Pr(\bL_k \supseteq P_1) &= \Pr(\bL_2 \supseteq P_2)\exp(O(n\log^2n)) \nonumber\\
    &= \left(\left(\frac{(1-\alpha)^{1-\alpha}(1-\beta)^{(1-\beta)^2}}{(1-\alpha-\beta)^{(1-\alpha-\beta)(1-\beta)}}\right)^{1/\alpha\beta}\frac{e^2}n(1+o(1))\right)^{\alpha\beta n^2}\exp(O(n\log^2n)) \nonumber\\
    &= \left(\left(\frac{(1-\alpha)^{1-\alpha}(1-\beta)^{(1-\beta)^2}}{(1-\alpha-\beta)^{(1-\alpha-\beta)(1-\beta)}}\right)^{1/\alpha\beta}\frac{e^2}n(1+o(1))\right)^{\alpha\beta n^2},\label{e:L_kP1}
    \end{align}
    since the $(1+o(1))^{\alpha\beta n^2}$ term can swallow the term $\exp(O(n\log^2n))$. The proposition now follows by combining \cref{e:symm_prob}, \cref{e:first_k_rows}, and \cref{e:L_kP1}.
\end{proof}

We now have all the tools that we need in order to prove \cref{t:Emn_asymptotics}.

\begin{proof}[Proof of \cref{t:Emn_asymptotics}]
    When $a = n/2$, McKay and Wanless~\cite{manysubsq} proved that
    \[
    \mathbb{E}_{a}(n) = 2^{-n^2}\exp(O(n\log^2n)) = \left(\frac{1+o(1)}8\right)^{a^2}\left(\frac an\right)^{a^2-3a},
    \]
    proving the theorem in this case. Henceforth, we assume that $a < n/2$.

    Let $T$ denote the set of all quadruples $(\cR, \cC, \cS, L)$ where $\cR, \cC, \cS \subseteq [n]$ are of cardinality $a$ and $L$ is a Latin square with row index set $\cR$, column index set $\cC$, and symbol set $\cS$. For $(\cR, \cC, \cS, L) \in T$, let $P = P(\cR, \cC, \cS, L)$ be the partial Latin square of order $n$ whose non-empty cells are precisely $\cR \times \cC$ that satisfies $P[\cR, \cC] = L$. Let $\bL$ be a random Latin square of order $n$. Then
	\begin{equation}\label{e:Eab_eq}
	\mathbb{E}_{a}(n) = \sum_{(\cR, \cC, \cS, L) \in T} \Pr(\bL \supseteq P(\cR, \cC, \cS, L)).
	\end{equation}
    By symmetry, 
    \begin{equation}\label{e:quad_prob}
    \Pr(\bL \supseteq P(\cR, \cC, \cS, L)) = \Pr(\bL \supseteq P(\cR', \cC', \cS', L'))
    \end{equation}
    for any $\{(\cR, \cC, \cS, L), (\cR', \cC', \cS', L')\} \subseteq T$. Thus, it suffices to bound $\Pr(\bL \supseteq P(\cR, \cC, \cS, L))$ for any fixed $(\cR, \cC, \cS, L) \in T$.
    There are $\binom{n}{a}$ possible choices each for $\cR$, $\cC$, and $\cS$ and by \cref{t:num_squares}, there are 
    \[
    \left(\frac a{e^2}(1+o(1))\right)^{a^2}
    \]
    choices for $L$ once $\cR$, $\cC$, and $\cS$ are determined. Fix $\cR = \cC = \cS = [a]$ and fix a Latin square $L$ of order $a$. Let $P = P(\cR, \cC, \cS, L)$. It follows from \cref{e:Eab_eq} that
    \begin{equation}\label{e:Ean_formula}
        \mathbb{E}_a(n) = \binom{n}{a}^3\left(\frac a{e^2}(1+o(1))\right)^{a^2}\Pr(\bL \supseteq P).
    \end{equation}

    We now distinguish two cases, depending on how $a$ grows with $n$.

    \textbf{Case 1:} $a = o(n)$. 
    By \cref{t:main},  
    \begin{equation}\label{e:PRCSL_prob}
    \left(\frac{\delta}n\right)^{a^2} \leq \Pr(\bL \supseteq P) \leq \left(\frac{\Delta}n\right)^{a^2}
    \end{equation}
    where $\delta = \delta(a/n, a/n) = 1/23-o(1)$ and $\Delta = \Delta(a/n, a/n) = 23+o(1)$ from \cref{t:main}. Combining 
    \cref{e:Ean_formula} with \cref{e:PRCSL_prob} and the inequality $\binom{n}{a} \geq (n/a)^a$, we obtain
    \[
    \mathbb{E}_a(n) \geq \left(\frac na\right)^{3a} \left(\frac a{e^2}(1+o(1))\right)^{a^2}\left(\frac{1-o(1)}{23n}\right)^{a^2} = \left(\frac{1-o(1)}{23e^2}\right)^{a^2}\left(\frac an\right)^{a^2-3a}.
    \]
    Similarly, combining \cref{e:Ean_formula} with \cref{e:PRCSL_prob} and the inequality $\binom{n}{a} \leq (en/a)^a$, we obtain
    \begin{align*}
    \mathbb{E}_a(n) &\leq \left(\frac {en}a\right)^{3a} \left(\frac a{e^2}(1+o(1))\right)^{a^2}\left(\frac{23+o(1)}{n}\right)^{a^2}  \\
    &= \left(\frac{23+o(1)}{e^{2-3/a}}\right)^{a^2}\left(\frac an\right)^{a^2-3a} \\
    & \leq  \left(\frac{23+o(1)}{e^{1/2}}\right)^{a^2}\left(\frac an\right)^{a^2-3a},
    \end{align*}
    where for the last inequality we have used the fact that $a \geq 2$.

    \textbf{Case 2:} $a = \Omega(n)$. Let $\alpha = \alpha(n) = a/n$. By \cref{prop:rectangles}, 
    \begin{equation}\label{e:contain_prob_omegan}
    \Pr(\bL \supseteq P) = \left(e^2\left(\frac{(1-\alpha)^{2-\alpha}}{(1-2\alpha)^{1-2\alpha}}\right)^{(1-\alpha)/\alpha^2}(1+o(1))\frac1n\right)^{a^2}.
    \end{equation}
    Define
    \[
    x = \left(\frac{(1-\alpha)^{2-\alpha}}{(1-2\alpha)^{1-2\alpha}}\right)^{(1-\alpha)/\alpha^2}
    \]
    and note that 
    \[
    \frac{1-o(1)}{23e^2} \leq x \leq \frac{23+o(1)}{e^2},
    \]
    regardless of $a$ (see \cref{fig:calpha} in \cref{s:conc}). Then by combining \cref{e:contain_prob_omegan} with \cref{e:Ean_formula} as in the previous case we can prove that
    \[
    \mathbb{E}_a(n) = (x(1+o(1)))^{a^2}\left(\frac an \right)^{a^2-3a},
    \]
    which proves the theorem.
\end{proof}

As alluded to in the introduction, we could have used~\cite[Theorem $1.6$]{subsqrandom} along with \cref{t:ext} in the range $\omega(n^{1/2}\log n) = a = o(n)$ in the above proof to get matching constants in the upper and lower bounds.
We are now ready to prove \cref{t:Eabn}, whose proof follows along similar lines as the proof of \cref{t:Emn_asymptotics}.

\begin{proof}[Proof of \cref{t:Eabn}]
    If $a=b$, then the theorem follows immediately from \cref{t:Emn_asymptotics}, thus we henceforth assume that $a \neq b$.
    Let $\epsilon > 0$, let $\alpha = \alpha(n) = a/n$, and let $\beta = \beta(n) = b/n$. Let $T$ denote the set of all quadruples $(\cR, \cC, \cS, L)$ where $\cR \subseteq [n]$ is of cardinality $a$, $\cC, \cS \subseteq [n]$ are of cardinality $b$, and $L$ is a Latin rectangle with row index set $\cR$, column index set $\cC$, and symbol set $\cS$. For $(\cR, \cC, \cS, L) \in T$, let $P = P(\cR, \cC, \cS, L)$ be the partial Latin square of order $n$ whose non-empty cells are precisely $\cR \times \cC$ that satisfies $P[\cR, \cC] = L$. Let $\bL$ be a random Latin square of order $n$. Then
	\begin{equation}\label{e:Eab_eq2}
	\mathbb{E}_{a, b}(n) = \sum_{(\cR, \cC, \cS, L) \in T} \Pr(\bL \supseteq P(\cR, \cC, \cS, L)).
	\end{equation}
    By symmetry, 
    \begin{equation}\label{e:quad_prob2}
    \Pr(\bL \supseteq P(\cR, \cC, \cS, L)) = \Pr(\bL \supseteq P(\cR', \cC', \cS', L'))
    \end{equation}
    for any $\{(\cR, \cC, \cS, L), (\cR', \cC', \cS', L')\} \subseteq T$. Thus, it suffices to bound the probability in \cref{e:quad_prob2} for any fixed $(\cR, \cC, \cS, L) \in T$.
    There are $\binom{n}{a}$ possible choices for $\cR$, there are $\binom{n}{b}$ choices each for $\cC$ and $\cS$, and there are $\mathscr{L}(a, b)$ choices for $L$ once $\cR$, $\cC$, and $\cS$ are determined. Fix $\cR = [a]$, fix $\cC = \cS = [b]$, and fix an $a \times b$ Latin rectangle $L$. Let $P = P(\cR, \cC, \cS, L)$. It follows from \cref{e:Eab_eq2} that
    \begin{equation}\label{e:Ean_formula2}
        \mathbb{E}_{a, b}(n) = \binom{n}{a}\binom{n}{b}^2\mathscr{L}(a, b)\Pr(\bL \supseteq P).
    \end{equation}

    We now distinguish two cases, depending on how $a$ grows with $n$.

    \textbf{Case 1:} $a < \epsilon n/2$. In this case, $b \leq (1-\epsilon)n$ and $2a+b \leq (1-\epsilon)n+\epsilon n/2 \leq (1-\epsilon/2)n$. Thus, we can employ \cref{t:main} to conclude that
    \begin{equation}\label{e:C1_prob}
    \left(\frac \delta n\right)^{ab} \leq \Pr(\bL \supseteq P) \leq \left(\frac \Delta n\right)^{ab}
    \end{equation}
    where $\delta = \delta(\epsilon/2, 1-\epsilon)$ and $\Delta = \Delta(\epsilon/2, 1-\epsilon)$ from \cref{t:main}.     
    By combining \cref{e:Ean_formula2}, \cref{e:C1_prob}, \cref{e:Lab}, and the fact that $\binom{n}{k} \geq (n/k)^k$ for any $k \in [n]$, we obtain
    \[
    \mathbb{E}_{a, b}(n) \geq \left(\frac na\right)^a \left(\frac nb \right)^{2b} \left(\frac b{e^2}\right)^{ab} \left(\frac \delta n\right)^{ab} = \left(\frac bn\right)^{ab-2b-a}\left(\frac ba\right)^{a}\left(\frac\delta{e^2}\right)^{ab}.
    \]
    Similarly, by combining \cref{e:Ean_formula2}, \cref{e:C1_prob}, \cref{e:Lab}, and the fact that $\binom{n}{k} \leq (en/k)^k$ for any $k \in [n]$, we obtain
    \[
    \mathbb{E}_{a, b}(n) \leq \left(\frac {en}a\right)^a \left(\frac {en}b \right)^{2b} b^{ab} \left(\frac \Delta n\right)^{ab} = \left(\frac bn\right)^{ab-2b-a}\left(\frac ba\right)^{a}\left(\Delta e^{(a+2b)/ab}\right)^{ab}.
    \]

    \textbf{Case 2:} $a \geq \epsilon n/2$. So $a, b = \Omega(n)$ and thus by \cref{prop:rectangles}, we can determine that
    \begin{equation}\label{e:C2_prob}
    \Pr(\bL \supseteq P) = \left(\frac {xe^2(1+o(1))}n\right)^{ab}
    \end{equation}
    where
    \[
    x = \left(\frac{(1-\alpha)^{1-\alpha}(1-\beta)^{(1-\beta)^2}}{(1-\alpha-\beta)^{(1-\alpha-\beta)(1-\beta)}}\right)^{1/\alpha\beta}.
    \]
    By \cref{t:num_rect},
    \begin{equation}\label{e:C2_numsq}
    \mathscr{L}(a, b) = \left(\frac{y(1+o(1))}{e^2}b\right)^{ab}
    \end{equation}
    where
    \[
    y = \left(\frac\beta{\beta-\alpha}\right)^{(\beta-\alpha)/\alpha}.
    \]
    Thus, by combining \cref{e:Ean_formula2}, \cref{e:C2_prob}, \cref{e:C2_numsq}, and the fact that $\binom{n}{k} = (n/k)^k\exp(O(k))$, we obtain
    \begin{align*}
    \mathbb{E}_{a, b}(n) &= \left(\frac na\right)^a\left(\frac nb\right)^{2b} \exp(O(b))\left(\frac{y(1+o(1))}{e^2}b\right)^{ab}\left(\frac {xe^2(1+o(1))}n\right)^{ab} \\
    &= \left(\frac bn\right)^{ab-2b-a}\left(\frac ba\right)^{a}(xy(1+o(1)))^{ab}.
    \end{align*}
    Finally, one has that 
    \[
    \frac18 < xy < 1
    \]
   for all values of  $a$ and $b$, with the infimum achieved as $\alpha,\beta\rightarrow 1/2$ and the supremum as $\alpha \rightarrow 0$ and $\beta \rightarrow 1$, which  can be checked by computer. This proves the theorem. 
\end{proof}

\subsection{The existence of isotopic copies} \label{ss:copies}

In this section, we prove Proposition \ref{prop:partial} which states that we a.a.s.\ find an isotopic copy of a constant-sized partial Latin square $P$ in a random Latin square $\bL$, as long as the density $m(P)$ is greater than $1/2$.  We begin with a definition. A partial Latin square $P$ is said to be \textit{1-degenerate} if there is some ordering of the entries of $P$ as $P=\{q_1,\ldots,q_{|P|}\}$ with $q_i=(r_i,c_i,s_i)$ for $1\leq i\leq |P|$ such that for each $1\leq i\leq |P|$:
\begin{itemize}
    \item either $r_i\neq r_j$ for all $1\leq j\leq i-1$; 
    \item or $c_i\neq c_j$ for all $1\leq j\leq i-1$; 
    \item or $s_i\neq s_j$ for all $1\leq j\leq i-1$. 
\end{itemize}
In words, for each $i=1,\ldots,|P|$, the entry $q_i$ has a unique row, column or symbol in the partial Latin square defined by $\{q_1,\ldots,q_i\}$. 

A simple consequence of Lemma \ref{l:newrcs} gives matching upper and lower bounds for 1-degenerate partial Latin squares. 

\begin{lem} \label{l:degen}
    Let $P$ be a $1$-degenerate partial Latin square of order $n$. Then for $\bL$ a random Latin square of order $n$ we have 
    \[ 
		\Pr(\mathbf{L}\supseteq P)=(1+o(1))(1/n)^{|P|}.
		\]
\end{lem}
\begin{proof}
    We let $P=\{q_1,\ldots,q_{|P|}\}$ with the ordering of entries witnessing the $1$-degeneracy of $P$. Let    $P_0=\emptyset$ and for $i=1,\ldots, |P|$, let $P_i:=\{q_1,\ldots,q_{i}\}$. By the chain rule of probability,
	\begin{equation*}
		\Pr(\bL \supseteq P) = \prod_{i=1}^{|P|} \Pr(\bL \supseteq P_i | \bL \supseteq P_{i-1}).
	\end{equation*}
	For each $1\leq i\leq |P|$, we then have that  $\Pr(\bL \supseteq P_i | \bL \supseteq P_{i-1})=(1+o(1))/n$ by Lemma \ref{l:newrcs} (and Remark~\ref{rem:symm}) on account of $q_i$ having a unique row, column or symbol in $P_i$. The result thus follows. 
\end{proof}

We are now ready to prove Proposition \ref{prop:partial}.
\begin{proof}[Proof of Proposition \ref{prop:partial}]
    Let $P$ be a partial Latin square with $m(P)>1/2$. We claim that $P$ is 1-degenerate. Indeed, we can form the required ordering of $P$ in reverse. Suppose that for some $1\leq i \leq |P|$, we have defined $q_j$ for $i<j\leq |P|$ and let $P':=P\setminus \{q_j:i<j\leq |P|\}$. We claim that there is some entry $q'\in P'$ which has a unique row, column or symbol. Indeed if this were not the case then 
    \[2(|\cR_{P'}|+|\cC_{P'}| +|\cS_{P'}|)\leq |\{(x,q)\in  (\cR_{P'}\times P' )\cupdot (\cC_{P'} \times P')\cupdot (\cS_{P'}\times P'): x\in q\}|=3|P'|,\]
    and so $|\cR_{P'}|+|\cC_{P'}| +|\cS_{P'}|-|P'|\leq |P'|/2$, contradicting that $m(P)>1/2$. Hence there does exist a $q'$ which has a unique row, column or symbol. We fix $q_i:=q'$ and repeating this for all $i=|P|,\ldots,1$ defines the necessary sequence. 

   Let $\cP$ be the collection of isotopic copies of $P$ and let $\sigma_P$ be the size of the autotopism group of $P$, that is the number of ways to permute $\cR_P$, $\cC_P$ and $\cS_P$ that result in obtaining $P$ again. We have that 
   \begin{equation} \label{eq:expectation est}
  |\cP|=\frac{n(n-1)\cdots(n-|\cR_P|+1)n\cdots(n-|\cC_P|+1)n\cdots(n-|\cS_P|+1)}{\sigma_P}=\frac{(1+o(1))}{\sigma_P}n^{|\cR_{P}|+|\cC_{P}| +|\cS_{P}|},      
   \end{equation}
   accounting for the choices for each non-empty row, column and symbol and correcting for the choices that give the same isotopic copy (the $\sigma_P$ factor).  We have used the fact that $|\cR_P|$, $|\cC_P|$, and $|\cS_P|$ are constant in \cref{eq:expectation est}.  
   By Lemma \ref{l:degen}, as $P$ is $1$-degenerate, we have that each copy of $P$ in $\cP$ lies in $\bL$ with probability $(1+o(1))/n^{|P|}$ (since $1$-degeneracy is easily seen to be preserved under taking isotopic copies). Hence, letting $X:=\sum_{P_0\in \cP}\mathbbm{1}[{\bL\supseteq  P_0}]$ be the random variable counting the number of copies of $P$ in $\bL$,   by linearity of expectation we get that \[\mathbb{E}[X]=\frac{(1+o(1))}{\sigma_P}n^{|\cR_{P}|+|\cC_{P}| +|\cS_{P}|-|P|}.\]
   By Chebyshev's inequality (see for example \cite[Section 1.2]{JLRbook}), we have that 
   \[\Pr(X=0)\leq \frac{\operatorname{Var}(X)}{\mathbb{E}[X]^2}=\frac{\mathbb{E}[X^2]-\mathbb{E}[X]^2}{\mathbb{E}[X]^2}.\]
   Hence it suffices to show that $\mathbb{E}[X^2]=(1+o(1))\mathbb{E}[X]^2$. For this we partition $\cP\times \cP$ as follows: 
\begin{itemize}
    \item Let $\cQ_1:=\{(P_1,P_2)\in \cP\times \cP: P_1\cap P_2\neq \emptyset\}$ be all pairs that share entries; 
    \item Let $\cQ_2:=\{(P_1,P_2)\in \cP\times \cP: (\cR_{P_1}\cap \cR_{P_2})\cup (\cC_{P_1}\cap \cC_{P_2}) \cup(\cS_{P_1}\cap \cS_{P_2}) \neq \emptyset\}\setminus \cQ_1$ be all pairs that do not share any entries but do share some non-empty row, column or symbol.
    \item Let $\cQ_3:=(\cP\times \cP)\setminus (\cQ_1\cup \cQ_2)$ be all the remaining pairs.
\end{itemize}
 For $(P_1,P_2)\in \cQ_1$, we have that $P_1\cap P_2$ is an isotopic copy of some partial Latin square $P'$ that is contained in $P$.   For $\emptyset\neq P'\subseteq P$, let $\cQ_1(P')$ be all the pairs $(P_1,P_2)\in \cQ_1$ whose intersection $P_1\cap P_2$ is isotopic to $P'$ and let $R(P')=2|\cR_{P}|-|\cR_{P'}|$, $C(P')=2|\cC_{P}|-|\cC_{P'}|$ and $S(P')=2|\cS_{P}|-|\cS_{P'}|$. 
 Note that for any $(P_1,P_2)\in \cQ_1(P')$ we have that $|\cR_{P_1\cup P_2}|=R(P')$
 and similarly for unions of the column and symbol sets of $P_1\cup P_2$. Moreover $|P_1\cup P_2|=2|P|-|P'|$ and 
as $|\cR_{P'}|+|\cC_{P'}|+|\cS_{P'}|-|P'| > |P'|/2\geq 1/2$ due to the definition of $m(P)$, we have that  
\begin{align} \nonumber R(P')+C(P')+S(P')-(2|P|-|P'|)&=2(|\cR_{P}|+|\cC_{P}| +|\cS_{P}|-|P|)-(|\cR_{P'}|+|\cC_{P'}| +|\cS_{P'}|-|P'|) \\ &\leq 2(|\cR_{P}|+|\cC_{P}| +|\cS_{P}|-|P|)-1/2. \label{eq:union est}\end{align}
    We can then bound the contribution of pairs in $\cQ_1$ to the variance as
     \begin{align}
    \nonumber     \mathbb{E}\left[ \sum_{(P_1,P_2)\in \cQ_1}\mathbbm{1}[\bL\supseteq P_1\cup P_2] \right] &= \sum_{\emptyset\neq P'\subseteq P}\quad \sum_{(P_1,P_2)\in \cQ_1(P') } \Pr(\bL\supseteq P_1\cup P_2) \nonumber  \\
     &\leq  \sum_{\emptyset\neq P'\subseteq P} \quad \sum_{(P_1,P_2)\in \cQ_1(P') } \left(24/n\right)^{|P_1\cup P_2|} \nonumber  \\
    &\leq \sum_{\emptyset\neq P'\subseteq P} n^{ R(P')+C(P')+S(P')}\left(24/n\right)^{2|P|-|P'|} \nonumber  \\
      &\leq 24^{2|P|} n^{2(|\cR_{P}|+|\cC_{P}| +|\cS_{P}|-|P|)-1/2}=o(\mathbb{E}[X]^2),
     \end{align} 
     where we used Theorem \ref{t:main} in the first inequality and \eqref{eq:expectation est} and \eqref{eq:union est} in the final line.

     Similarly, we have that 
     \begin{equation} \label{eq:Q2 estimate}
    \mathbb{E}\left[ \sum_{(P_1,P_2)\in \cQ_2}\mathbbm{1}[\bL\supseteq P_1\cup P_2] \right]=\sum_{(P_1,P_2)\in \cQ_2 } \Pr(\bL\supseteq P_1\cup P_2)\leq \sum_{(P_1,P_2)\in \cQ_2 } (24/n)^{2|P|}, \end{equation}
    using Theorem \ref{t:main} and the fact that $\cQ_2\cap\cQ_1=\emptyset$. As $|\cQ_2|=O(n^{2(|\cR_{P}|+|\cC_{P}| +|\cS_{P}|)-1})$, we also get that the expectation in equation \eqref{eq:Q2 estimate} is  $o(\mathbb{E}[X]^2)$ by considering \eqref{eq:expectation est}. 
    Finally then we have that 
    \begin{align*}
       \mathbb{E}[X^2]&=\mathbb{E}\left[\sum_{(P_1,P_2)\in \cP^2}\mathbbm{1}[\bL\supseteq P_1]\mathbbm{1}[\bL\supseteq P_2]\right] 
      = \mathbb{E}\left[\sum_{(P_1,P_2)\in \cP^2}\mathbbm{1}[\bL\supseteq P_1\cup P_2]\right]
       \\ &= \sum_{(P_1,P_2)\in \cQ_3}\Pr(\bL\supseteq P_1\cup P_2) +o(\mathbb{E}[X]^2)
       \\ &= \sum_{(P_1,P_2)\in \cQ_3} (1+o(1))(1/n)^{2|P|} +o(\mathbb{E}[X]^2)
       \\ &= (1+o(1))\frac{n(n-1)\cdots(n-2|\cR_P|+1)n\cdots(n-2|\cC_P|+1)n\cdots(n-2|\cS_P|+1)}{\sigma_P^2}(1/n)^{2|P|} +o(\mathbb{E}[X]^2)
       \\&=(1+o(1))\mathbb{E}[X]^2
    \end{align*}
    as required. Here we used Lemma \ref{l:degen} to bound the probability that $\bL\supseteq P_1\cup P_2$. This was possible as $P_1\cup P_2$ is $1$-degenerate for $(P_1,P_2)\in \cQ_3$ due to the fact that both $P_1$ and $P_2$ are  1-degenerate and they do not share any rows, columns or symbols. This completes the proof.  
\end{proof}

\section{Future work}\label{s:conc}

We have already mentioned two interesting directions for future research in \cref{ss:tightness}: Improving the constants $\delta$ and $\Delta$ in \cref{t:main} and investigating under which other conditions on partial Latin squares a theorem of the same nature as \cref{t:main} could hold. The former direction suggests the following interesting question: Given a partial Latin square $P$ and a random Latin square $\bL$, does 
\begin{equation}\label{e:lim}
\lim_{n \to \infty} n(\Pr(\bL \supseteq P))^{1/|P|}
\end{equation}
exist, and if so, what is it? \cref{prop:rectangles} tells us that when the non-empty cells of $P$ induce a Latin square of order $a = \alpha n$ with $0<\alpha<1/2$, this limit does exist and is equal to
\[
c(\alpha) := e^2\left(\frac{(1-\alpha)^{2-\alpha}}{(1-2\alpha)^{1-2\alpha}}\right)^{(1-\alpha)/\alpha^2}.
\]

\begin{figure}[h]
		\centering
		\includegraphics[width=0.64\linewidth]{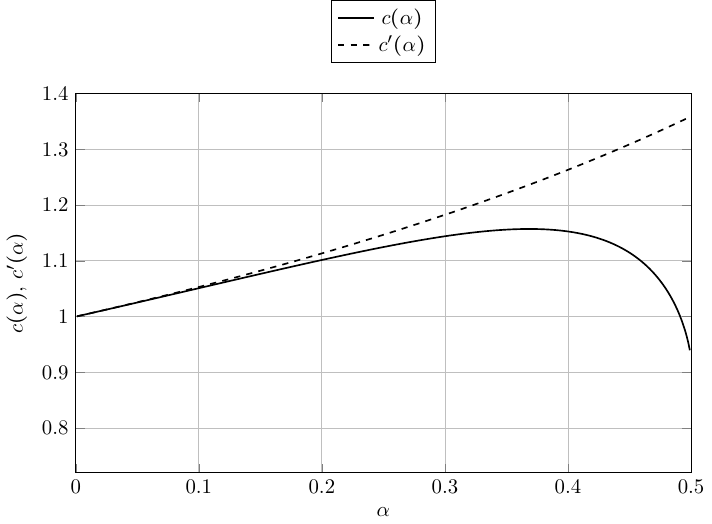}
		\caption{The functions $c(\alpha)$ and $c'(\alpha)$ for $0<\alpha<1/2$.}
		\label{fig:calpha}
	\end{figure} 
\cref{fig:calpha} plots the function $c(\alpha)$ in the range $0 < \alpha < 1/2$.
The behaviour of $c(\alpha)$ demonstrated by \cref{fig:calpha} is interesting and perhaps somewhat surprising. Indeed, if you were to model a random Latin square $\bL$ by sampling a uniformly random permutation independently for each row (a model which successfully predicts many properties of random Latin squares cf.\ \cite{LSparity}), then one would expect the probability of containing a fixed subsquare of order $\alpha n$ to be $((n-\alpha n)!/n!)^{\alpha n}$ which would result in a  monotone increasing value for $c(\alpha)$ (in fact this  prediction would give the constant $c'(\alpha)$ defined below).

We also record the fact that the limit in \cref{e:lim}, if it exists, cannot just depend on $|\cR_P|/n$ and $|\cC_P|/n$ (although we expect this to be the case if $|\cR_P|, |\cC_P| = o(n)$). Indeed, if $P$ is a partial Latin square whose non-empty cells induce a Latin square of order $a = \alpha n$, then this limit is equal to $c(\alpha)$ discussed above. Define $P'$ to be the partial Latin square of order $n$ defined by:
\begin{itemize}
    \item $P'[1, i] = i$ for all $i \in [a]$,
    \item $P'[i, 1] = i$ for all $i \in [a]$,
    \item all other cells of $P'$ are empty.
\end{itemize}
Using the ideas behind~\cite[Lemma $3.2$]{subsqrandomrect}, it is simple to prove that for a random Latin square $\bL$ of order $n$, 
\[
\Pr(\bL \supseteq P') = n\left(\frac{(n(1-\alpha))!}{n!}\right)^2,
\]
and so 
\[
c'(\alpha):=\lim_{n \to \infty} n(\Pr(\bL \supseteq P'))^{1/|P'|} = e(1-\alpha)^{(1-\alpha)/\alpha},
\]
which is different from $c(\alpha)$. Both functions $c(\alpha)$ and $c'(\alpha)$ are plotted in Figure~\ref{fig:calpha}. The fact that these functions differ justifies the claim that this limit in \cref{e:lim}, if it exists, must depend on the structure of $P$ in a non-trivial way. It is an interesting challenge to even pose a plausible conjecture regarding when this limit exists and what value it takes, which of course depends on the partial Latin square $P$. A similar challenge was posed in the thesis of the first author~\cite{thesis}.

A further direction for future research is to extend our results to the setting of Latin rectangles. Latin rectangles are much more flexible objects than Latin squares and consequently there has been more progress on studying substructures in `thin' Latin rectangles than there has previously been for Latin squares.  
 For $K>0$, a $k \times n$ Latin rectangle is \emph{$K$-sparse} if every row contains at most $K$ non-empty cells, every column contains at most $K$ non-empty cells, and every symbol occurs at most $K$ times. Recently, improving upon a result of Godsil and McKay~\cite{GM90}, Divoux, Kelly, Kennedy, and Sidhu~\cite{subsqrandom} proved that if $P$ is a $k \times n$ partial Latin rectangle with $k \leq (1/2-o(1))n$ that is $o(1)$-sparse, then the probability that a random $k \times n$ Latin rectangle contains $P$ is $n^{-|P|}\exp(o(|P|))$. They conjectured that the same result should hold for all $k \leq n$. \cref{t:main} resolves their conjecture when $k=n$ up to constant factors when $|P|$ is constant. However, it sheds no light on the conjecture when $k<n$. As mentioned in \cref{ss:eta_switch}, the main obstacle to extending our results to Latin rectangles is the fact that $\eta$-switching is not well-defined in a Latin rectangle. Generalising \cref{t:main} to the setting of Latin rectangles is a worthy direction for future research.

 We conclude by discussing one last direction for future research. A probability distribution $\cP$ on the space of Latin squares of order $n$ is $q$-spread if for \textit{every} partial Latin square $P$ of order $n$, the probability that a random Latin square $\bL \sim \cP$ chosen from this distribution contains $P$ is at most $q^{|P|}$. There has been recent interest in obtaining an $O(1/n)$-spread distribution on the space of Latin squares of order $n$, due to its connection with the so-called threshold problem for Latin squares. Indeed, via breakthrough results on the (fractional) Kahn-Kalai conjecture \cite{spread_threshold} the existence of such a spread distribution implies the correct threshold for a random 3-uniform 3-partite hypergraph to contain a Latin square. The first meaningful result in this direction was by Sah, Sawhney, and Simkin~\cite{thresh1}, who constructed an $n^{o(1)-1}$-spread distribution on the set of Latin squares of order $n$. Shortly after, Kang, Kelly, K\"uhn, Methuku, and Osthus~\cite{thresh2} proved the existence of a $O(\log n/n)$-spread distribution. Finally, Jain and Pham~\cite{thresh3}, and independently Keevash~\cite{thresh4}, proved the existence of a $O(1/n)$-spread distribution and thus settled the threshold problem. None of these distributions were the uniform distribution. Nonetheless, it is natural to believe that the uniform distribution is $O(1/n)$-spread, which would provide an alternative direct proof to determining the threshold of a Latin square. Although \cref{t:main} does not prove that the uniform measure is $O(1/n)$-spread, since it does not apply to all partial Latin squares of order $n$, 
it does show the spread condition for a wide range of partial Latin squares and gives  evidence to believe that the statement is true.
 In fact, a conjecture of Kelly~\cite{KellyConj} even predicts the correct constant in the spreadness, stating that the uniform distribution on the space of Latin squares of order $n$ is $(e^2(1+o(1))/n)$-spread. Note that the constant $e^2$ cannot be improved to any smaller constant, as can be seen by setting $P$ to be a (full) Latin square.  

\printbibliography

@article {subsqrandom,
    AUTHOR = {Divoux, A. and Kelly, T. and Kennedy, C. and Sidhu, J.},
     TITLE = {Subsquares in {R}andom {L}atin {S}quares and {R}ectangles},
   JOURNAL = {J. Combin. Des.},
  FJOURNAL = {Journal of Combinatorial Designs},
    VOLUME = {34},
      YEAR = {2026},
    NUMBER = {4},
     PAGES = {184--197},
      ISSN = {1063-8539,1520-6610},
   MRCLASS = {05B15},
  MRNUMBER = {5034585},
       DOI = {10.1002/jcd.70004},
       URL = {https://doi.org/10.1002/jcd.70004},
}

@article {manysubsq,
	AUTHOR = {McKay, B. D. and Wanless, I. M.},
	TITLE = {Most {L}atin squares have many subsquares},
	JOURNAL = {J. Combin. Theory Ser. A},
	FJOURNAL = {Journal of Combinatorial Theory. Series A},
	VOLUME = {86},
	YEAR = {1999},
	NUMBER = {2},
	PAGES = {322--347},
	ISSN = {0097-3165,1096-0899},
	MRCLASS = {05B15},
	MRNUMBER = {1685535},
	MRREVIEWER = {J.\ D\'{e}nes},
	DOI = {10.1006/jcta.1998.2947},
	URL = {https://doi.org/10.1006/jcta.1998.2947},
}

@article {cycstrucrandom,
	AUTHOR = {Cavenagh, N. J. and Greenhill, C. and Wanless,
		I. M.},
	TITLE = {The cycle structure of two rows in a random {L}atin square},
	JOURNAL = {Random Structures Algorithms},
	FJOURNAL = {Random Structures \& Algorithms},
	VOLUME = {33},
	YEAR = {2008},
	NUMBER = {3},
	PAGES = {286--309},
	ISSN = {1042-9832,1098-2418},
	MRCLASS = {05B15},
	MRNUMBER = {2446483},
	MRREVIEWER = {Steven\ T.\ Dougherty},
	DOI = {10.1002/rsa.20216},
	URL = {https://doi.org/10.1002/rsa.20216},
}

@article {subsqrandomrect,
	AUTHOR = {Allsop, J. and Wanless, I. M.},
	TITLE = {Subsquares in random {L}atin rectangles},
	JOURNAL = {Combinatorica},
	FJOURNAL = {Combinatorica. An International Journal on Combinatorics and
		the Theory of Computing},
	VOLUME = {45},
	YEAR = {2025},
	NUMBER = {3},
	Note = {Paper No. 29, 18pp.},
	ISSN = {0209-9683,1439-6912},
	MRCLASS = {05B15},
	MRNUMBER = {4905568},
	DOI = {10.1007/s00493-025-00156-0},
	URL = {https://doi.org/10.1007/s00493-025-00156-0},
}

@article {KS,
	AUTHOR = {Kwan, M. and Sudakov, B.},
	TITLE = {Intercalates and discrepancy in random {L}atin squares},
	JOURNAL = {Random Structures Algorithms},
	FJOURNAL = {Random Structures \& Algorithms},
	VOLUME = {52},
	YEAR = {2018},
	NUMBER = {2},
	PAGES = {181--196},
	ISSN = {1042-9832,1098-2418},
	MRCLASS = {05B15 (60C05)},
	MRNUMBER = {3758956},
	MRREVIEWER = {Lars-Daniel\ \"{O}hman},
	DOI = {10.1002/rsa.20742},
	URL = {https://doi.org/10.1002/rsa.20742},
}

@article {KSS,
	AUTHOR = {Kwan, M. and Sah, A. and Sawhney, M.},
	TITLE = {Large deviations in random {L}atin squares},
	JOURNAL = {Bull. Lond. Math. Soc.},
	FJOURNAL = {Bulletin of the London Mathematical Society},
	VOLUME = {54},
	YEAR = {2022},
	NUMBER = {4},
	PAGES = {1420--1438},
	ISSN = {0024-6093,1469-2120},
	MRCLASS = {60F10 (05B15 05C80)},
	MRNUMBER = {4488316},
	MRREVIEWER = {Oliver\ Johnson},
	DOI = {10.1112/blms.12638},
	URL = {https://doi.org/10.1112/blms.12638},
}

@article {KSSS,
	AUTHOR = {Kwan, M. and Sah, A. and Sawhney, M. and Simkin,
	M.},
	TITLE = {Substructures in {L}atin squares},
	JOURNAL = {Israel J. Math.},
	FJOURNAL = {Israel Journal of Mathematics},
	VOLUME = {256},
	YEAR = {2023},
	NUMBER = {2},
	PAGES = {363--416},
	ISSN = {0021-2172,1565-8511},
	MRCLASS = {05B15 (60C05)},
	MRNUMBER = {4651013},
	DOI = {10.1007/s11856-023-2513-9},
	URL = {https://doi.org/10.1007/s11856-023-2513-9},
}

@phdthesis{thesis,
	author  = "Allsop, J.",
	title   = "Latin subrectangles",
	school  = "Monash University",
	year    = "2025",
	url     = {https://bridges.monash.edu/articles/thesis/Latin_Subrectangles/30302224?file=58551064}
}

@article {canonicallabel,
	AUTHOR = {Gill, M. J. and Mammoliti, A. and Wanless, I. M.},
	TITLE = {Canonical labeling of {L}atin squares in average-case
	polynomial time},
	JOURNAL = {Random Structures Algorithms},
	FJOURNAL = {Random Structures \& Algorithms},
	VOLUME = {66},
	YEAR = {2025},
	NUMBER = {4},
	note = {Paper No. e70015, 23pp.},
	ISSN = {1042-9832,1098-2418},
	MRCLASS = {05B15},
	MRNUMBER = {4931615},
	DOI = {10.1002/rsa.70015},
	URL = {https://doi.org/10.1002/rsa.70015},
}

@article {GM90,
	AUTHOR = {Godsil, C. D. and McKay, B. D.},
	TITLE = {Asymptotic enumeration of {L}atin rectangles},
	JOURNAL = {J. Combin. Theory Ser. B},
	FJOURNAL = {Journal of Combinatorial Theory. Series B},
	VOLUME = {48},
	YEAR = {1990},
	NUMBER = {1},
	PAGES = {19--44},
	ISSN = {0095-8956,1096-0902},
	MRCLASS = {05A16 (05B15)},
	MRNUMBER = {1047551},
	DOI = {10.1016/0095-8956(90)90128-M},
	URL = {https://doi.org/10.1016/0095-8956(90)90128-M},
}

@incollection {genswitch,
	AUTHOR = {Hasheminezhad, M. and McKay, B. D.},
	TITLE = {Combinatorial estimates by the switching method},
	BOOKTITLE = {Combinatorics and graphs},
	SERIES = {Contemp. Math.},
	VOLUME = {531},
	PAGES = {209--221},
	PUBLISHER = {Amer. Math. Soc., Providence, RI},
	YEAR = {2010},
	ISBN = {978-0-8218-4865-4},
	MRCLASS = {05A16 (05A20 05C20 60C05)},
	MRNUMBER = {2757801},
	MRREVIEWER = {Pu\ Gao},
	DOI = {10.1090/conm/531/10469},
	URL = {https://doi.org/10.1090/conm/531/10469},
}

@article {cycswitch,
	AUTHOR = {Wanless, I. M.},
	TITLE = {Cycle switches in {L}atin squares},
	JOURNAL = {Graphs Combin.},
	FJOURNAL = {Graphs and Combinatorics},
	VOLUME = {20},
	YEAR = {2004},
	NUMBER = {4},
	PAGES = {545--570},
	ISSN = {0911-0119,1435-5914},
	MRCLASS = {05B15},
	MRNUMBER = {2108400},
	MRREVIEWER = {Haitao\ Cao},
	DOI = {10.1007/s00373-004-0567-7},
	URL = {https://doi.org/10.1007/s00373-004-0567-7},
}

@article {genLS,
	AUTHOR = {Jacobson, M. T. and Matthews, P.},
	TITLE = {Generating uniformly distributed random {L}atin squares},
	JOURNAL = {J. Combin. Des.},
	FJOURNAL = {Journal of Combinatorial Designs},
	VOLUME = {4},
	YEAR = {1996},
	NUMBER = {6},
	PAGES = {405--437},
	ISSN = {1063-8539,1520-6610},
	MRCLASS = {05B15 (60J10)},
	MRNUMBER = {1410617},
	MRREVIEWER = {Lars\ D\o vling\ Andersen},
	DOI = {10.1002/(SICI)1520-6610(1996)4:6<405::AID-JCD3>3.0.CO;2-J},
	URL =
	{https://doi.org/10.1002/(SICI)1520-6610(1996)4:6<405::AID-JCD3>3.0.CO;2-J},
}

@article {LSparity,
    author = {Kwan, M. and Petrova, K. and Sawhney, M.},
    title = {Parities in random Latin squares},
    year = {2025},
    journal = {arXiv:2509.13125},
}

@article{ryser1951combinatorial,
  title={A combinatorial theorem with an application to {L}atin rectangles},
  author={Ryser, H. J.},
  journal={Proc. Amer. Math. Soc.},
  volume={2},
  number={4},
  pages={550--552},
  year={1951},
  publisher={JSTOR}
}

@book {JLRbook,
  AUTHOR =       {Janson, S. and {\L}uczak, T. and Ruci{\'n}ski, A.},
  TITLE =        {Random graphs},
  PUBLISHER =    {Wiley-Interscience},
  YEAR =         2000,
}

@article{kwan2020almost,
  title={Almost all {S}teiner triple systems have perfect matchings},
  author={Kwan, M.},
  JOURNAL = {Proc. Lond. Math. Soc. (3)},
  volume={121},
  number={6},
  pages={1468--1495},
  year={2020},
  publisher={Wiley Online Library}
}

@article {trade_survey,
    AUTHOR = {Cavenagh, N. J.},
     TITLE = {The theory and application of {L}atin bitrades: a survey},
   JOURNAL = {Math. Slovaca},
  FJOURNAL = {Mathematica Slovaca},
    VOLUME = {58},
      YEAR = {2008},
    NUMBER = {6},
     PAGES = {691--718},
      ISSN = {0139-9918,1337-2211},
   MRCLASS = {05B15},
  MRNUMBER = {2453264},
MRREVIEWER = {G.\ H. J. Van Rees},
       DOI = {10.2478/s12175-008-0103-2},
       URL = {https://doi.org/10.2478/s12175-008-0103-2},
}

@article {num_rect,
    AUTHOR = {Luria, Z. and Simkin, M.},
     TITLE = {On the threshold problem for {L}atin boxes},
   JOURNAL = {Random Structures Algorithms},
  FJOURNAL = {Random Structures \& Algorithms},
    VOLUME = {55},
      YEAR = {2019},
    NUMBER = {4},
     PAGES = {926--949},
      ISSN = {1042-9832,1098-2418},
   MRCLASS = {94A60 (94A62)},
  MRNUMBER = {4025395},
       DOI = {10.1002/rsa.20855},
       URL = {https://doi.org/10.1002/rsa.20855},
}

@book {num_squares,
    AUTHOR = {van Lint, J. H. and Wilson, R. M.},
     TITLE = {A course in combinatorics},
   EDITION = {Second},
 PUBLISHER = {Cambridge University Press, Cambridge},
      YEAR = {2001},
     PAGES = {xiv+602},
      ISBN = {0-521-00601-5},
   MRCLASS = {05-01 (90B10)},
  MRNUMBER = {1871828},
       DOI = {10.1017/CBO9780511987045},
       URL = {https://doi.org/10.1017/CBO9780511987045},
}

@article {Pitt,
    AUTHOR = {Pittenger, A. O.},
     TITLE = {Mappings of {L}atin squares},
   JOURNAL = {Linear Algebra Appl.},
  FJOURNAL = {Linear Algebra and its Applications},
    VOLUME = {261},
      YEAR = {1997},
     PAGES = {251--268},
      ISSN = {0024-3795,1873-1856},
   MRCLASS = {05B15},
  MRNUMBER = {1448875},
MRREVIEWER = {Charles\ C.\ Lindner},
       URL =
              {http://www.sciencedirect.com/science?_ob=GatewayURL&_origin=MR&_method=citationSearch&_piikey=s0024379596004089&_version=1&md5=8f77dde5dd9117f120df735f9e7bef97},
}

@article {completion_survey,
    AUTHOR = {Donovan, D.},
     TITLE = {The completion of partial {L}atin squares},
   JOURNAL = {Australas. J. Combin.},
  FJOURNAL = {The Australasian Journal of Combinatorics},
    VOLUME = {22},
      YEAR = {2000},
     PAGES = {247--264},
      ISSN = {1034-4942,2202-3518},
   MRCLASS = {05B15},
  MRNUMBER = {1795340},
MRREVIEWER = {G.\ H. J. Van Rees},
}

@article {thresh1,
    AUTHOR = {Sah, A. and Sawhney, M. and Simkin, M.},
     TITLE = {Threshold for {S}teiner triple systems},
   JOURNAL = {Geom. Funct. Anal.},
  FJOURNAL = {Geometric and Functional Analysis},
    VOLUME = {33},
      YEAR = {2023},
    NUMBER = {4},
     PAGES = {1141--1172},
      ISSN = {1016-443X,1420-8970},
   MRCLASS = {05B07 (60C05)},
  MRNUMBER = {4616696},
MRREVIEWER = {Luc\ Teirlinck},
       DOI = {10.1007/s00039-023-00639-6},
       URL = {https://doi.org/10.1007/s00039-023-00639-6},
}

@article {thresh2,
    AUTHOR = {Kang, D. Y. and Kelly, T. and K\"uhn, D. and
              Methuku, A. and Osthus, D.},
     TITLE = {Thresholds for {L}atin squares and {S}teiner triple systems:
              bounds within a logarithmic factor},
   JOURNAL = {Trans. Amer. Math. Soc.},
  FJOURNAL = {Transactions of the American Mathematical Society},
    VOLUME = {376},
      YEAR = {2023},
    NUMBER = {9},
     PAGES = {6623--6662},
      ISSN = {0002-9947,1088-6850},
   MRCLASS = {05B15 (05B05 05B07 05C80)},
  MRNUMBER = {4630786},
MRREVIEWER = {Carl\ Johan\ Casselgren},
       DOI = {10.1090/tran/8954},
       URL = {https://doi.org/10.1090/tran/8954},
}

@inproceedings {thresh3,
    AUTHOR = {Jain, V. and Pham, H. T.},
     TITLE = {Optimal thresholds for {L}atin squares, {S}teiner triple
              systems, and edge colorings},
 BOOKTITLE = {Proceedings of the 2024 {A}nnual {ACM}-{SIAM} {S}ymposium on
              {D}iscrete {A}lgorithms ({SODA})},
     PAGES = {1425--1436},
 PUBLISHER = {SIAM, Philadelphia, PA},
      YEAR = {2024},
      ISBN = {978-1-61197-791-2},
   MRCLASS = {68Q87},
  MRNUMBER = {4699302},
       DOI = {10.1137/1.9781611977912.57},
       URL = {https://doi.org/10.1137/1.9781611977912.57},
}

@article {spread_threshold,
    AUTHOR = {Frankston, K. and Kahn, J. and Narayanan, B. and
              Park, J.},
     TITLE = {Thresholds versus fractional expectation-thresholds},
   JOURNAL = {Ann. of Math. (2)},
  FJOURNAL = {Annals of Mathematics. Second Series},
    VOLUME = {194},
      YEAR = {2021},
    NUMBER = {2},
     PAGES = {475--495},
      ISSN = {0003-486X,1939-8980},
   MRCLASS = {05C80 (60C05 82B26)},
  MRNUMBER = {4298747},
MRREVIEWER = {Tatyana\ S.\ Turova},
       DOI = {10.4007/annals.2021.194.2.2},
       URL = {https://doi.org/10.4007/annals.2021.194.2.2},
}

@article {thresh4,
    title = {The optimal edge-colouring threshold},
    author = {Keevash, P.},
    journal = {arXiv:2212.04397},
    year = {2022},
}

@misc{KellyConj,
    author = {Kelly, T.},
    note = {“Nibble Methods in Graph Theory.” In Topics in Probabilistic Graph Theory, edited by R. J. Wilson, L. W. Beineke, and
C. McDiarmid. Cambridge University Press. Forthcoming}
}

@article{decomp,
    author = {Bowtell, C. and Montgomery, R.},
    title = {Almost every Latin square has a decomposition into transversals},
    journal = {arXiv:2501.05438},
    year = {2025},
}

\end{document}